\documentclass[]{gOMS2e}

\usepackage{mathtools}
\usepackage{psfrag}
\usepackage{rotating}
\usepackage[usenames]{color}
\usepackage{graphicx}
\graphicspath{{figures/}}
\usepackage{epsfig}    
\usepackage{amssymb,scalerel}   
\usepackage[font=footnotesize,caption=false]{subfig}
\usepackage{fixltx2e}
\usepackage{array}
\usepackage{dblfloatfix}

\usepackage{arydshln}

\usepackage{url}
\usepackage{tikz}

\usepackage{algorithm}     
\usepackage{algpseudocode} 

\DeclareMathOperator*{\minimize}{minimize}

\DeclareMathOperator*{\minimum}{minimum}
\DeclareMathOperator*{\subject}{subject\ to}

\DeclareMathOperator*{\maximum}{max}
\DeclareMathOperator*{\argmin}{argmin}

\DeclareMathOperator*{\diag}{diag}

\DeclareMathOperator*{\parent}{par}

\DeclareMathOperator*{\children}{ch}

\DeclareMathOperator*{\rank}{rank}
\DeclareMathOperator*{\blkdiag}{blk\ diag}
\DeclareMathOperator*{\leaves}{leaves}
\DeclareMathOperator*{\Ne}{Ne}

\newcommand{\Bigblacksquare}{\mathord{\includegraphics[height=8ex]{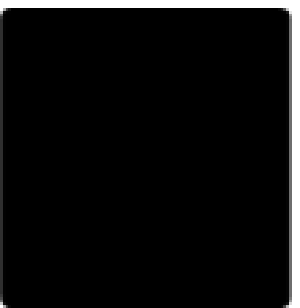}}}
\newcommand{\bigblacksquare}{\mathord{\includegraphics[height=6ex]{square.eps}}}

\newcounter{remcount}
\newtheorem{rem}[remcount]{Remark}

\definecolor{red}{rgb}{1,0,0}

\begin{document}
\doi{10.1080/1055.6788.YYYY.xxxxxx}
\issn{1029-4937}
\issnp{1055-6788}
\jvol{00} \jnum{00} \jyear{2014} \jmonth{October}

\markboth{Taylor \& Francis and I.T. Consultant}{Optimization Methods and Software}

\title{Distributed Primal-dual Interior-point Methods for Solving Loosely Coupled Problems Using Message Passing$^*$\thanks{$^*$This work has been supported by the Swedish Department of Education within the ELLIIT project.}}


\author{Sina Khoshfetrat Pakazad$^{1}$\thanks{$^{1}$ Sina Khoshfetrat Pakazad and Anders Hansson are with the Division of Automatic Control, Department of Electrical Engineering, Link\"oping University, Sweden. Email: \{sina.kh.pa, hansson\}@isy.liu.se.}, Anders Hansson$^{1}$ and Martin S. Andersen$^{2}$ \thanks{$^{2}$Martin S. Andersen is with the Department of Applied Mathematics and Computer Science, Technical University of Denmark. Email:  mskan@dtu.dk.}}

\maketitle


\begin{abstract}
In this paper, we propose a distributed algorithm for solving loosely coupled problems with chordal sparsity which relies on primal-dual interior-point methods. We achieve this by distributing the computations at each iteration, using message-passing. In comparison to already existing distributed algorithms for solving such problems, this algorithm requires far less number of iterations to converge to a solution with high accuracy. Furthermore, it is possible to compute an upper-bound for the number of required iterations which, unlike already existing methods, \emph{only} depends on the coupling structure in the problem. We illustrate the performance of our proposed method using a set of numerical examples.

\begin{classcode} \end{classcode}\bigskip

\end{abstract}

\section{Introduction}\label{sec:Intro}

Centralized algorithms for solving optimization problems rely on the existence of a central computational unit powerful enough to solve the problem in a timely manner, and they render useless in case we lack such a unit. Also such algorithms become unviable when it is impossible to form the problem in a centralized manner, for instance due to structural constraints including privacy requirements. In cases like these, distributed optimization algorithms are the only resort for solving optimization problems, e.g., see \cite{ber:97,eck:89,boyd:11,ned:09,ned:10}. In this paper we are interested in devising efficient distributed algorithms for solving convex optimization problems in the form
\begin{subequations}\label{eq:EOP}
\begin{align}
\minimize   \quad &f_1(x)+ \dots + f_N(x)\\
\subject \quad & G^i(x) \preceq 0, \quad i=1, \dots, N, \\
  & A^i x = b^i, \hspace{6mm} i=1, \dots, N,
\end{align}
\end{subequations}
where $f_i \ : \ \mathbb R^{n}\rightarrow \mathbb R$, $G^i \ : \ \mathbb R^{n}\rightarrow \mathbb R^{m_i}$, $A^i \in \mathbb R^{p_i \times n}$ with $\sum_{i=1}^{N} p_i < n$ and $\rank(A^i) = p_i$ for all $i = 1, \dots, N$. Here $\preceq$ denotes the component-wise inequality. This problem can be seen as a combination of $N$ coupled subproblems, each of which is defined by an objective function $f_i$ and by constraints that are expressed by $G^i$ and matrices $A^i$ and $b^i$. Furthermore, we assume that these subproblems are only dependent on a few elements of $x$, and that they are loosely coupled. The structure in such problems is a form of partial separability, which implies that the Hessian of the problem is sparse, see e.g., \cite{sun:14} and references therein.  Existing distributed algorithms for solving \eqref{eq:EOP}, commonly solve the problem using a computational network with $N$ computational agents, each of which is associated with its own local subproblem. The graph describing this computational network has the node set $V = \{1, \dots, N \}$ with an edge between any two nodes in case they need to communicate with one another. The existence of an edge also indicates existence of coupling among the subproblems associated to neighboring agents. This graph is referred to as the \emph{computational} graph of the algorithm and matches the coupling structure in the problem, which enables us to solve the problem distributedly while providing complete privacy among the agents.

Among different algorithms for solving problems like \eqref{eq:EOP} distributedly, the ones based on first order methods are among the simplest ones. These algorithms are devised by applying gradient/subgradient or proximal point methods to the problem or an equivalent reformulations of it, see e.g., \cite{ned:09,ned:10,ber:97,eck:89,boyd:11,com:11}. In this class, algorithms that are based on gradient or subgradient methods, commonly require simple local computations. However, they are extremely sensitive to the scaling of the problem, see e.g., \cite{ned:09,ned:10}. Algorithms based on proximal point methods alleviate the scaling sensitivity issue, see e.g., \cite{ber:97,eck:89,boyd:11,com:11}, but this comes at a price of more demanding local computations and/or more sophisticated communication protocols among agents, see e.g., \cite{sum:12,ohl:13,gol:12,gols:12}.

Despite the effectiveness of this class of algorithms, they generally still require many iterations to converge to an accurate solution. In order to improve the convergence properties of the aforementioned algorithms, there has recently been a surge of interest in devising distributed algorithms using second order methods, see e.g., \cite{Chu:11,wei:13,nec:09,kho:14,ang:14}. In \cite{nec:09}, the authors propose a distributed optimization algorithm based on a Lagrangian dual decomposition technique which enables them to use second order information of the dual function to update the dual variables within a dual interior-point framework. To this end, at each iteration, every agent solves a constrained optimization problem for updating local primal variables and then communicates with all the other agents to attain the necessary information for updating the dual variables. This level of communication is necessary due to the coupling in the considered optimization problem. The authors in \cite{wei:13} present a distributed Newton method for solving a network utility maximization problem. The proposed method relies on the special structure in the problem, which is that the objective function is given as a summation of several decoupled terms, each of which depends on a single variable. This enables them to utilize a certain matrix splitting method for computing Newton directions distributedly. In \cite{kho:14,ang:14} the authors put forth distributed primal and primal-dual interior-point methods that rely on proximal splitting methods, particularly ADMM, for solving for primal and primal-dual directions, distributedly. This then allows them to propose distributed implementations of their respective interior-point methods. One of the major advantages of the proposed algorithms in \cite{wei:13,kho:14,ang:14} lies in the fact that the required local computations are very simple. These approaches are based on inexact computations of the search directions, and they rely on first order or proximal methods to compute these directions. Generally the number of required iterations to compute the directions depends on the desired accuracy, and in case they require high accuracy for the computed directions, this number can grow very large. This means that commonly the computed directions using these algorithms are not accurate, and particularly the agents only have approximate consensus over the computed directions. This inaccuracy of the computed directions can also sometimes adversely affect the number of total primal or primal-dual iterations for solving the problem.


In this paper we propose a distributed primal-dual interior-point method and we evade the aforementioned issues by investigating  another distributed approach to solve for primal-dual directions. To this end we borrow ideas from so-called message-passing algorithms for exact inference over probabilistic graphical models, \cite{kol:09,pea:82}. In this class of inference methods, message-passing algorithms are closely related to non-serial dynamic programming, see e.g., \cite{ber:73,moa:07,wai:05,kol:09}. Non-serial dynamic programming techniques, unlike serial dynamic programming, \cite{ber:00}, that are used for solving problems with chain-like or serial coupling structure, are used to solve problems with general coupling structure. Specifically, a class of non-serial dynamic programming techniques utilize a tree representation of the coupling in the problem and use similar ideas as in serial techniques to solve the problem efficiently, see e.g., \cite{ber:73,wai:05,moa:07,shc:07}. We here also use a similar approach for computing the primal-dual directions. As we will see later, this enables us to devise distributed algorithms, that unlike the previous ones compute the exact directions within a finite number of iterations. In fact, this number can be computed a priori, and it only depends on the coupling structure in the problem. Unfortunately these advantages come at a cost. Particularly, these algorithms can only be efficiently applied to problems that are sufficiently sparse. Furthermore, for these algorithms the computational graphs can differ from the coupling structure of the problem, and hence they can only provide partial privacy among the agents. The approach presented in this paper is also closely related to multi-frontal factorization techniques for sparse matrices, e.g., see \cite{liu:92,duf:83,and:13}. In fact we will show that the message-passing framework can be construed as a distributed multi-frontal factorization method using fixed pivoting for certain sparse symmetric indefinite matrices. To the best knowledge of the authors the closest approach to the one put forth in this paper is the work presented in \cite{gon:07,gon:09}. The authors for these papers, propose an efficient primal-dual interior-point method for solving problems with a so-called nested block structure. Specifically, by exploiting this structure, they  present an efficient way for computing primal-dual directions by taking advantage of parallel computations when computing factorization of the coefficient matrix in the augmented system at each iteration. In this paper, we consider a more general coupling structure and focus on devising a distributed algorithm for computing the search directions, and we provide assurances that this can be done even when each agent has a limited access to information regarding the problem, due to privacy constraints.

\subsection*{Outline}
Next we first define some of the common notations used in this paper, and in Section~\ref{sec:CPP} we put forth a general description of coupled optimization problems and describe mathematical and graphical ways to express the coupling in the problem. In Section \ref{sec:chordal} we review some concepts related to chordal graphs. These are then used in Section~\ref{sec:OMP} to describe distributed optimization algorithms based on message-passing for solving coupled optimization problems. We briefly describe the primal-dual interior-point method in Section \ref{sec:PDIPM}. In Section \ref{sec:DPDIPM}, we first provide a formal mathematical description for loosely coupled problems and then we show how primal-dual methods can be applied in a distributed fashion for solving loosely coupled problems. Furthermore, in this section we discuss how the message-passing framework is related to multi-frontal factorization techniques. We test the performance of the algorithm using a numerical example in Section~\ref{sec:number}, and finish the paper with some concluding remarks in Section \ref{sec:conclude}.

\subsection*{Notation}
We denote by $\mathbb R$ the set of real scalars and by $\mathbb R^{n\times m}$ the set of real $n\times m$ matrices. With $\mathbf 1$ we denote a column vector of all ones. The set of $n \times n$ symmetric matrices are represented by $\mathbb S^n$. The transpose of a matrix $A$ is denoted by $A^T$ and the column and null space of this matrix is denoted by $\mathcal{C}(A)$ and $\mathcal N(A)$, respectively. We denote the set of positive integers $\{1,2,\ldots,p\}$ with $\mathbb{N}_p$. Given a set $J \subset \mathbb{N}_n$, the matrix $E_J \in \mathbb{R}^{|J|\times n}$ is the $0$-$1$ matrix that is obtained by deleting the rows indexed by $\mathbb{N}_n \setminus J$ from an identity matrix of order $n$, where $|J|$ denotes the number of elements in set $J$. This means that $E_Jx$ is a $|J|$- dimensional vector with the components of $x$ that correspond to the elements in $J$, and we denote this vector with $x_J$. With $x^{i,(k)}_l$ we denote the $l$th element of vector $x^i$ at the $k$th iteration. Also given vectors $x^i$ for $i= 1, \dots, N$, the column vector $(x^1, \dots, x^N)$ is all of the given vectors stacked.

\section{Coupled Optimization Problems}\label{sec:CPP}

Consider the following convex optimization problem
\begin{align}\label{eq:CP}
\minimize_{x} \quad F_1(x) + \dots + F_N(x),
\end{align}
where $F_i \ : \ \mathbb R^{n}\rightarrow \mathbb R$ for all $i = 1, \dots, N$. We assume that each function $F_i$ is only dependent on a small subset of elements of $x$. Particularly, let us denote the ordered set of these indices by $J_i \subseteq \mathbb N_n$. We also denote the ordered set of indices of functions that depend on $x_i$ with $\mathcal I_i = \{ k \ | \ i \in J_k\} \subseteq \mathbb N_N$. With this description of coupling within the problem, we can now rewrite the problem in~\eqref{eq:CP}, as
\begin{align}\label{eq:CPS}
\minimize_{x} & \quad   \bar F_1(E_{J_1}x) + \dots + \bar F_N(E_{J_N}x),
\end{align}
where $E_{J_i}$ is a $0$--$1$ matrix that is obtained from an identity matrix of order $n$ by deleting the rows indexed by $\mathbb{N}_n \setminus J_i$.  The functions $\bar F_i \ : \ \mathbb R^{|J_i|} \rightarrow \mathbb R$ are lower dimensional descriptions of $F_i$s such that $F_i(x) = \bar F_i(E_{J_i}x)$ for all $x \in \mathbb R^n$ and $i = 1, \dots, N$. For instance consider the following optimization problem
\begin{align}
\minimize_x \quad F_1(x) + F_2(x) +  F_3(x) + F_4(x) + F_5(x) + F_6(x),
\end{align}
and let us assume that $x \in \mathbb R^8$, $J_1 = \{ 1, 3 \}$, $J_2 = \{ 1, 2, 4 \}$, $J_3 = \{ 4, 5 \}$, $J_4 = \{ 3, 4 \}$, $J_5 = \{ 3, 6, 7 \}$ and $J_6 = \{ 3, 8 \}$. With this dependency description we then have $\mathcal I_1 = \{ 1, 2 \}$, $\mathcal I_2 = \{ 2 \}$, $\mathcal I_3 = \{ 1, 4, 5, 6 \}$, $\mathcal I_4 = \{ 2, 3, 4 \}$, $\mathcal I_5 = \{ 3 \}$, $\mathcal I_6 = \{ 5 \}$, $\mathcal I_7 = \{ 5 \}$  and $\mathcal I_8 = \{ 6 \}$.  This problem can then be written in the same format as in \eqref{eq:CPS} as
\begin{multline}\label{eq:example}
\minimize_x \quad \bar F_1(x_1,x_3) + \bar F_2(x_1, x_2, x_4) +\\ \bar F_3(x_4, x_5) + \bar F_4(x_3, x_4) + \bar F_5(x_3, x_6, x_7) + \bar F_6(x_3, x_8).
\end{multline}
The formulation of coupled problems as in \eqref{eq:CPS} enables us to get a more clear picture of the coupling in the problem. Next we describe how the coupling structure in~\eqref{eq:CP} can be expressed graphically using undirected graphs.
\subsection{Coupling and Sparsity Graphs}


A graph $G$ is specified by its vertex and edge sets $V$ and $\mathcal E$, respectively.  The coupling structure in \eqref{eq:CP} can be described using an undirected graph with node or vertex set $V_c = \{ 1, \dots, N \}$ and the edge set $\mathcal{E}_c$ with $(i,j) \in \mathcal E_c$ if and only if $J_i \cap J_j \neq \emptyset$. We refer to this graph, $G_c$, as the \emph{coupling graph} of the problem. Notice that all sets $\mathcal I_i$ induce complete subgraphs on the coupling graph of the problem. Another graph that sheds more light on the coupling structure of the problem is the so-called \emph{sparsity graph}, $G_s$, of the problem. This graph is also undirected, though with node or vertex set $V_s = \{ 1, \dots, n \}$ and the edge set $\mathcal{E}_s$ with $(i,j) \in \mathcal E_s$ if and only if $\mathcal I_i \cap \mathcal I_j \neq \emptyset$. Similarly, all sets $J_i$ induce complete subgraphs on the sparsity graph of the problem. Let us now reconsider the example in~\eqref{eq:example}. The sparsity and coupling graphs for this problem are illustrated in Figure \ref{fig:SG}, on the left and right respectively. It can then be verified that all $J_i$s and $\mathcal I_i$s induce complete graphs over coupling and sparsity graphs, respectively.
\begin{figure}[t]
\begin{center}
\includegraphics[width=8cm]{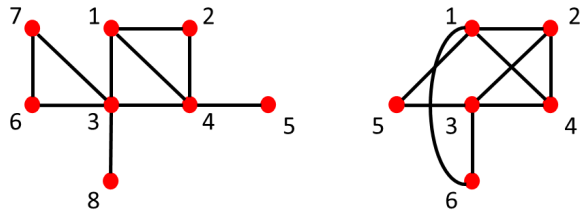}    
\caption{\small The sparsity and coupling graphs for the problem in~\eqref{eq:example}.\normalsize }
\label{fig:SG}
\end{center}
\end{figure}

As we will see later graph representations of the coupling structure in problems play an important role in designing distributed algorithms for solving coupled problems and gaining insight regarding their distributed implementations. Specifically, chordal graphs and their characteristics play a major role in the design of our proposed algorithm. This is the topic of the next section.


\section{Chordal Graphs}\label{sec:chordal}

A graph $G(V,\mathcal E)$ with vertex set $V$ and edge set $\mathcal E$ is chordal if every of its cycles of length at least four has a chord, where a chord is an edge between two non-consecutive vertices in a cycle, \cite[Ch. 4]{gol:04}. A clique of $G$ is a \emph{maximal} subset of $V$ that induces a complete subgraph on $G$. Consequently, no clique of $G$ is entirely contained in any other clique, \cite{blp:94}. Let us denote the set of cliques of $G$ as $\mathbf C_G = \{ C_1, \dots, C_q \}$. There exists a tree defined on $\mathbf C_G $ such that for every $C_i, C_j \in \mathbf C_G $ with $i \neq j$, $C_i \cap C_j$ is contained in all the cliques in the path connecting the two cliques in the tree. This property is called the clique intersection property, and trees with this property are referred to as clique trees. For instance the graph on the left in Figure \ref{fig:SG} is chordal and has five cliques, namely $C_1 = \{ 1, 2, 4 \}$, $C_2 = \{ 1, 3, 4 \}$, $C_3 = \{ 4, 5 \}$, $C_4 = \{ 3, 6, 7 \}$ and $C_5 = \{ 3, 8 \}$. A clique tree over these cliques is given in Figure \ref{fig:SC}. This tree then satisfies the clique intersection property, e.g., notice that $C_2 \cap C_3 = \{ 4 \}$ and the only clique in the path between $C_2$ and $C_3$, that is $C_1$, also includes $\{ 4 \}$.

Chordal graphs and their corresponding clique trees play a central role in the design of the upcoming algorithms. For chordal graphs there are efficient methods for computing cliques and clique trees. However, the graphs that we will encounter, particularly the sparsity graphs, do not have to be chordal. As a result, next and for the sake of completeness we first review simple heuristic methods to compute a chordal embedding of such graphs, where a chordal embedding of a graph $G(V,\mathcal E)$ is a chordal graph with the same vertex set and an edge set $\mathcal E_e$ such that $\mathcal E \subseteq \mathcal E_e$. We will also explain how to compute its cliques and the corresponding clique tree.

\subsection{Chordal Embedding and Its Cliques}

Greedy search methods are commonly used for computing chordal embeddings of graphs, where one such method is presented in Algorithm~\ref{alg:embed}, \cite{cor:01}, \cite{kol:09}. The graph $G$  with the returned edge set $\mathcal E$ will then be a chordal graph.
\begin{algorithm}[tb]
\caption{Greedy Search Method for Chordal Embedding}\label{alg:embed}
\begin{algorithmic}[1]
\small
\State{Given a graph $G(V,\mathcal E)$ with $V = \{ 1, \dots, n\}$, $\mathbf C_G  = \emptyset$, $V_t = V$, $\mathcal E_t = \mathcal E$ and $flag = 1$}
\Repeat
\State{$i = $ vertex in $V_t$ with the smallest number of neighbors based on $\mathcal E_t$}
\State{Connect all the nodes in $\text{Ne}(i)$ to each other and add the newly generated edges to $\mathcal E_t$ and $\mathcal E$}
\State{$C_t = \{i\} \cup \text{Ne}(i)$}
\State{$\mathcal E_t = \mathcal E_t \setminus \left\{ (i,j) \in \mathcal E_t \ \big| \ j \in \text{Ne}(i)   \right\}$}
\State{$V_t = V_t \setminus \{i\}$}
\For{$k = 1 \ : \ |\mathbf C_G |$}
\If {$C_t \subseteq \mathbf C_G (k)$}
\State{$flag = 0$}
\EndIf
 \EndFor
\If {$flag$}
\State{$\mathbf C_G  = \mathbf C_G  \cup \{C_t\}$}
\EndIf
\State{$flag = 1$}
\Until {$V_t = \emptyset$}
\normalsize
\end{algorithmic}
\end{algorithm}
This algorithm also computes the set of cliques of the computed chordal embedding which are returned in the set $\mathbf C_G$. Notice that $\text{Ne}(i)$ in steps 4, 5 and 6 is defined based on the most recent description of the sets $V_t$ and $\mathcal E_t$. The criterion used in Step 3 of the algorithm for selecting a vertex is the so-called \emph{min-degree criterion}. There exist other versions of this algorithm that utilize other criteria, e.g., \emph{min-weight}, \emph{min-fill} and \emph{weighted-min-fill}. Having computed a chordal embedding of the graph and its clique set, we will next review how to compute a clique tree over the computed clique set.

\subsection{Clique Trees}

Assume that a set of cliques for a chordal graph $G$ is given as $\mathbf C_G  = \{ C_1, C_2, \dots, C_q \}$. In order to compute a clique tree over the clique set we need to first define a weighted undirected graph, $W$, over $V_W = \{ 1, \dots, q \}$ with edge set $\mathcal E_W$ where $(i, j) \in \mathcal E_W$ if and only if $C_i \cap C_j \neq \emptyset$, where the assigned weight to this edge is equal to $\big | C_i \cap C_j \big |$. A clique tree over $C_G$ can be computed by finding any maximum spanning tree of the aforementioned weighted graph. This means finding a tree in the graph that contains all its nodes and edges with maximal accumulated weight. An algorithm to find such a tree is presented in Algorithm~\ref{alg:span}, \cite{cor:01}, \cite{kol:09}. The tree described by the vertex set $V_t$ and edge set $\mathcal E_t$ is then a clique tree.
\begin{algorithm}[tb]
\caption{Maximum Weight Spanning Tree}\label{alg:span}
\begin{algorithmic}[1]
\small
\State{Given a weighted graph $W(V_W,\mathcal E_W)$ with $V_W = \{ 1, \dots, q\}$, $V_t = {1}$ and $\mathcal E_t = \emptyset$}
\Repeat
\State{$\mathcal E = \left\{ (i,j) \in \mathcal E_W \ \big| \ i \in V_t, j \notin V_t\right\}$}
\State{$(\bar i,\bar j) = (i,j) \in \mathcal E$ with the highest weight}
\State{$V_t = V_t \cup \{ \bar j \}$}
\State{$\mathcal E_t = \mathcal E_t \cup \{ (\bar i,\bar j)  \}$}
\Until {$V_t =  V_W$}
\normalsize
\end{algorithmic}
\end{algorithm}
We will now discuss distributed optimization using message-passing.
\section{Optimization Over Clique Trees}\label{sec:OMP}

In this section, we describe a distributed optimization algorithm based on message-passing. Particularly, we focus on the building blocks of this algorithm, namely we will provide a detailed description of its computational graph,  messages exchanged among agents, the communication protocol they should follow and how they compute their corresponding optimal solutions. The convergence and computational properties of such methods, within exact inference over probabilistic graphical models, are extensively discussed in \cite[Ch. 10, Ch. 13]{kol:09}. For the sake of completeness and future reference, we here also review some of these results and provide proofs for these results using the unified notation in this paper, in the appendix.
\subsection{Distributed Optimization Using Message-passing}

Consider the optimization problem in \eqref{eq:CP}. Let $G_s(V_s,\mathcal E_s)$ denote the chordal sparsity graph for this problem and let $\mathbf C_s = \{ C_1, \dots, C_q  \}$ and $T(V_t, \mathcal E_t)$ be its set of cliques and a corresponding clique tree, respectively. It is possible to devise a distributed algorithm for solving this problem that utilizes the clique tree $T$ as its computational graph. This means that the nodes $V_t = \{ 1, \dots, q \}$ act as computational agents and collaborate with their neighbors that are defined by the edge set $\mathcal E_t$ of the tree. For example, the sparsity graph for the problem in \eqref{eq:example} has five cliques and a clique tree over these cliques is illustrated in Figure~\ref{fig:SC}. This means the problem can be solved distributedly using a network of five computational agents, each of which needs to collaborate with its neighbors as defined by the edges of the tree, e.g., Agent 2 needs to collaborate with agents $1, 4, 5$.

\begin{figure}[t]
\begin{center}
\includegraphics[width=7cm]{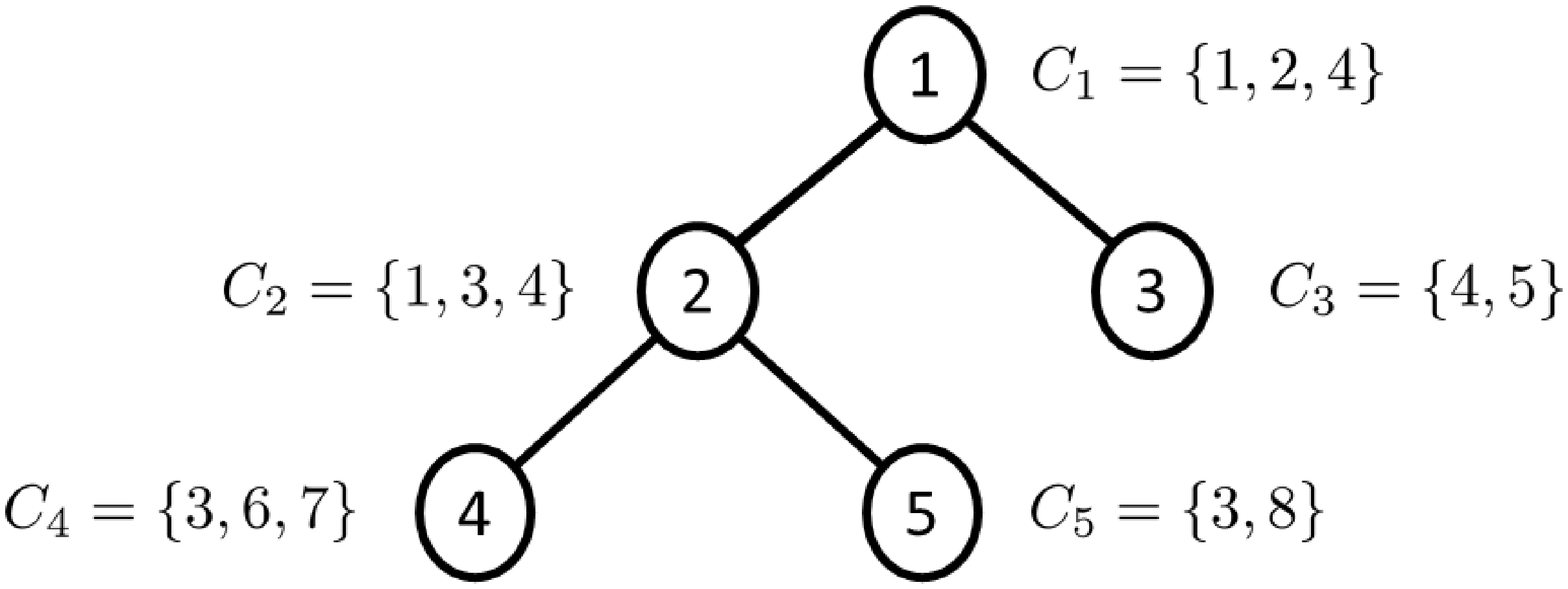}    
\caption{\small Clique tree for the sparsity graph of the problem \eqref{eq:example}.\normalsize }
\label{fig:SC}
\end{center}
\end{figure}

In order to specify the messages exchanged among these agents, we first assign different terms of the objective function in \eqref{eq:CP} to each agent. A valid assignment in this framework is that $F_i$ can only be assigned to agent $j$ if $J_i \subseteq C_j$. We denote the ordered set of indices of terms of the objective function assigned to agent $j$ by $\phi_j$. For instance, for the problem in \eqref{eq:example}, assigning $\bar F_1$ and $\bar F_4$ to Agent 2 would be a valid assignment since $J_1, J_4 \subseteq C_2$ and hence $\phi_2 = \{ 1, 4 \}$. Notice that the assignments are not unique and for instance there can exist agents $j$ and $k$ with $j \neq k$ so that $J_i \subseteq C_j$ and $J_i \subseteq C_k$ making assigning $F_i$ to agents $j$ or $k$ both valid. Also for every term of the objective function there will always exist an agent that it can be assigned to, which is proven in the following proposition.
\begin{proposition}
For each term $F_i$ of the objective function, there always exists a $C_j$ for which $J_i \subseteq C_j$.
\end{proposition}
\begin{proof}
Recall that each set $J_i$ induces a complete subgraph on the sparsity graph, $G_s$, of the problem. Then by definition of cliques, $J_i$ is either a subset of a clique or is a clique of the sparsity graph.
\end{proof}
\begin{figure}[t]
\begin{center}
\includegraphics[width=5.9cm]{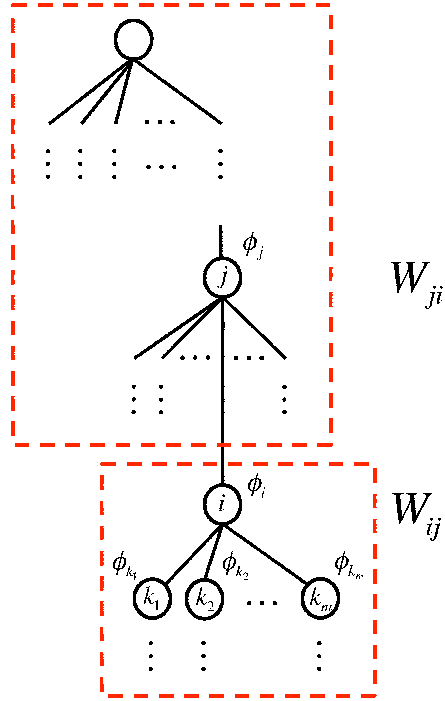}    
\caption{\small Clique tree for a sparsity graph $G_s$.\normalsize }
\label{fig:CT}
\end{center}
\end{figure}
Before we continue with the rest of the algorithm description, we first need to define some new notations that are going to be extensively used in the following. Consider Figure \ref{fig:CT} which illustrates a clique tree $T(V_t, \mathcal E_t)$ for a given sparsity graph $G_s$. Each node in the tree is associated to a clique of $G_s$ and let $W_{ij}$ denote the set of indices of cliques that are on the node $i$-side of edge $(i,j) \in \mathcal E_t$. Similarly, $W_{ji}$ denotes the same but for the ones on the $j$-side of $(i,j)$. Also we denote the set of indices of variables in the cliques specified by $W_{ij}$ by $V_{ij}$, i.e., $V_{ij} = \bigcup_{k \in W_{ij}} C_k$. Similarly the set of indices of variables in cliques specified by $W_{ji}$ is denoted by $V_{ji}$. The set of all indices of objective function terms that are assigned to nodes specified by $W_{ij}$ is represented by $\Phi_{ij}$, i.e., $\Phi_{ij} = \bigcup_{k \in W_{ij}} \phi_{k}$, and the ones specified by $W_{ji}$ with $\Phi_{ji}$. In order to make the newly defined notations more clear, let us reconsider the example in \eqref{eq:example} and its corresponding clique tree in Figure \ref{fig:SC}, and let us focus on the $(1,2)$ edge. For this example then $W_{21} = \{ 2, 4, 5 \}$, $W_{12} = \{ 1, 3 \}$, $V_{21} = \{ 1, 3, 4, 6, 7, 8 \}$, $V_{12} = \{ 1, 2, 4, 5 \}$, $\Phi_{21} = \{ 1, 4, 5, 6 \}$ and $\Phi_{12} = \{ 2, 3 \}$. With the notation defined, we will now express the messages that are exchanged among neighboring agents. Particularly, let $i$ and $j$ be two neighboring agents, then the message sent from agent $i$ to agent $j$, $m_{ij}$, is given by
\begin{align}\label{eq:mij}
m_{ij}(x_{_{S_{ij}}}) = \minimum_{x_{_{C_i \setminus S_{ij}}}} \left\{  \sum_{k \in \phi_i} \bar F_k(x_{_{J_k}}) +  \sum_{k \in \Ne(i)\setminus \{j\}} m_{ki}(x_{_{S_{ik}}}) \right\},
\end{align}
where $S_{ij} = C_i \cap C_j$ is the so-called separator set of agents $i$ and $j$. As a result, for agent $i$ to be able to send the correct message to agent $j$ it needs to wait until it has received all the messages from its neighboring agents other than $j$. Hence, the information required for computing a message also sets the communication protocol for this algorithm. Specifically, it sets the ordering of agents in the message-passing procedure in the algorithm, where messages can only be initiated  from the leaves of the clique tree and upwards to the root of the tree, which is referred to as an upward pass through the tree. For instance, for the problem in \eqref{eq:example} and as can be seen in Figure~\ref{fig:SC}, $\Ne(2) = \{ 1, 4, 5 \}$. Then the message to be sent from Agent 2 to Agent 1 can be written as
\begin{align}\label{eq:mijexample}
m_{21}(x_1, x_4) = \minimum_{x_3} \left\{  \bar F_1(x_1, x_3) + \bar F_4(x_3, x_4) +  m_{42}(x_3) + m_{52}(x_3) \right\}.
\end{align}
which can only be computed if Agent 2 has received the messages from agents 4 and 5.

The message, $m_{ij}$, that every agent $j$ receives from a neighboring agent $i$ in fact summarizes all the necessary information that agent $j$ needs from all the agents on the $i$-side of the edge $(i, j)$. Particularly this message provides the optimal value of
\begin{align*}
\sum_{t \in \Phi_{ij}}\bar F_t(x_{_{J_t}})
\end{align*}
as a function of the variables that agents $i$ and $j$ share, i.e., $x_{_{S_{ij}}}$. This is shown in the following theorem.
\begin{theorem}\label{thm:thm2}
Consider the message sent from agent $i$ to agent $j$ as defined in \eqref{eq:mij}. This message can also be equivalently rewritten as
\begin{align}\label{eq:thm2}
m_{ij} (x_{_{S_{ij}}}) = \minimum_{x_{_{V_{ij} \setminus S_{ij} }}} \left\{\sum_{t \in \Phi_{ij}}\bar F_t(x_{_{J_t}}) \right\}
\end{align}
\end{theorem}
\begin{proof}
See \cite[Thm. 10.3]{kol:09} or Appendix \ref{app:app3}.
\end{proof}
With this description of messages and at the end of an upward-pass through the clique tree, the agent at the root of the tree, indexed $r$, will have received messages from all its neighbors. Consequently, it will have all the necessary information to compute its optimal solution by solving the following optimization problem
\begin{align}\label{eq:RLocalProblem}
x^\ast_{_{C_r}} = \argmin_{x_{_{C_r}}} \left\{  \sum_{k \in \phi_r} \bar F_k(x_{_{J_k}})  + \sum_{k \in \Ne(r)} m_{kr}(x_{_{S_{rk}}}) \right\}.
\end{align}
The next theorem proves the optimality of such a solution.
\begin{theorem}\label{thm:thm3}
The equation in \eqref{eq:RLocalProblem} can be rewritten as
\begin{align}\label{eq:LocalProblem1}
x^\ast_{_{C_r}} = \argmin_{x_{_{C_r}}} \left\{  \minimum_{x_{_{\mathbb N_n \setminus C_r}}} \left\{ \bar F_1(x_{_{J_1}}) + \dots +\bar F_N(x_{_{J_N}})\right\}\right\},
\end{align}
which means that $x^\ast_{_{C_r}}$ denotes the optimal solution for elements of $x$ specified by $C_r$.
\end{theorem}
\begin{proof}
See \cite[Corr. 10.2, Prop. 13.1]{kol:09} or Appendix \ref{app:app4}
\end{proof}
Let us now assume that the agent at the root having computed its optimal solution $x^\ast_{_{C_r}}$, sends messages $m_{rj}(x_{_{S_{rj}}})$ and the computed optimal solution $\left(x^\ast_{_{S_{rj}}}\right)^r$ to its children, i.e., to all agents $j \in  \children(r)$. Here $\left(x^\ast_{_{S_{rj}}}\right)^r$ denotes the optimal solution computed by agent $r$. Then all these agents, similar to the agent at the root, will then have received messages from all their neighbors and can compute their corresponding optimal solution as
\begin{align}\label{eq:LocalProblemi}
x^\ast_{_{C_i}} = \argmin_{x_{_{C_i}}} \left\{  \sum_{k \in \phi_i} \bar F_k(x_{_{J_k}})  + \sum_{k \in \Ne(i)} m_{ki}(x_{_{S_{ik}}}) + \frac{1}{2} \left\| x_{_{S_{ri}}} - \left( x_{_{S_{ri}}}^\ast \right)^r \right\|^2 \right\}.
\end{align}
Notice that since $x^\ast_{_{C_r}}$ is optimal, the additional regularization term in \eqref{eq:LocalProblemi} will not affect the optimality of the solution. All it does is to assure that the computed optimal solution by the agent is consistent with that of the root. This observation also allows us to rewrite \eqref{eq:LocalProblemi} as

\small
\begin{align}\label{eq:LocalProblemiMod}
x^\ast_{_{C_i}} &= \argmin_{x_{_{C_i}}} \left\{  \sum_{k \in \phi_i} \bar F_k(x_{_{J_k}})  + \sum_{k \in \Ne(i)\setminus r} m_{ki}(x_{_{S_{ik}}}) + m_{ri}\left(\left( x_{_{S_{ri}}}^\ast \right)^r\right) + \frac{1}{2} \left\| x_{_{S_{ri}}} - \left( x_{_{S_{ri}}}^\ast \right)^r \right\|^2 \right\}\notag\\
&= \argmin_{x_{_{C_i}}} \left\{  \sum_{k \in \phi_i} \bar F_k(x_{_{J_k}})  + \sum_{k \in \Ne(i)\setminus r} m_{ki}(x_{_{S_{ik}}}) + \frac{1}{2} \left\| x_{_{S_{ri}}} - \left( x_{_{S_{ri}}}^\ast \right)^r \right\|^2 \right\}.
\end{align}
\normalsize
This means that the root does not need to compute nor send the message $m_{rj}(x_{_{S_{rj}}})$ to its neighbors and it suffices to only communicate its computed optimal solution. The same procedure is executed downward through the tree until we reach the leaves, where each agent $i$, having received the computed optimal solution by its parent, i.e., $\left( x_{_{S_{\parent(i)i}}}^\ast \right)^{\parent(i)}$, computes its optimal solution by
\begin{align}\label{eq:LocalProblempar}
x^\ast_{_{C_i}} = \argmin_{x_{_{C_i}}} \left\{  \sum_{k \in \phi_i} \bar F_k(x_{_{J_k}})   + \sum_{k \in \Ne(i)\setminus \parent(i)} m_{ki}(x_{_{S_{ik}}}) + \frac{1}{2} \left\| x_{_{S_{\parent(i)i}}} - \left( x_{_{S_{\parent(i)i}}}^\ast \right)^{\parent(i)} \right\|^2 \right\}.
\end{align}
where $\parent(i)$ denotes the index for the parent of agent $i$. As a result by the end of one upward-downward pass through the clique tree, all agents have computed their corresponding optimal solutions, and hence, at this point, the algorithm can be terminated. Furthermore, with this way of computing the optimal solution, it is always assured that the solutions computed by parents and the children  are consistent with one another. Since this is the case for all the nodes in the clique tree, it follows that we have consensus over the network. A summary of this distributed approach is given in Algorithm~\ref{alg:MP}.

\begin{algorithm}[H]
\caption{Distributed Optimization Using Message Passing}\label{alg:MP}
\begin{algorithmic}[1]
\small
\State{Given sparsity graph $G_s$ of an optimization problem}
\State{Compute a chordal embedding of $G_s$, its cliques and a clique tree over the cliques.}
\State{Assign each term of the objective function to one and only one of the agents.}
\State{Perform message passing upwards from the leaves to the root of the tree.}
\State{Perform a downward pass from the root to the leaves of the tree, where each agent, having received information about the optimal solution of its parent, computes its optimal solution using \eqref{eq:LocalProblempar} and communicates it to its children.}
\State By the end of the downward pass all agents have computed their optimal solutions and the algorithm is  terminated.
\end{algorithmic}
\end{algorithm}
\begin{rem}\label{rem:rem1}
\emph{Notice that in case the optimal solution of \eqref{eq:CP} is unique, then we can drop the regularization term in \eqref{eq:LocalProblempar} since the computed optimal solutions by the agents will be consistent due to the uniqueness of the optimal solution.}
\end{rem}
So far we have provided a distributed algorithm to compute a consistent optimal solution  for convex optimization problems in the form \eqref{eq:CP}. However, this algorithm relies on the fact that we are able to eliminate variables and compute the optimal objective value as a function of the remaining ones in closed form. This capability is essential, particularly for computing the exchanged messages among agents and in turn limits the scope of problems that can be solved using this algorithm. We will later show how the described algorithm can be incorporated within a primal-dual interior-point method to solve general convex optimization problems, distributedly.
\begin{rem}\label{rem:rem2}
\emph{ The message-passing scheme presented in this section is in fact a recursive algorithm and it terminates within a finite number of steps or after an upward-downward pass. Let us define, $L$, the height of a tree as the maximum number of edges in a path from the root to a leaf. This number then tells us how many steps it will take to perform the upward-downward pass through the tree. As a result, the shorter the tree the fewer the number of steps we need to take to complete a pass through the tree and compute the solution. Due to this fact, and since given a tree we can choose any node to be the root, having computed the clique tree we can improve the convergence properties of our algorithm by choosing a node as the root that gives us the minimum height.}
\end{rem}
\subsection{Modifying the Generation of the Computational Graph}
\begin{figure}[t]
\begin{center}
\includegraphics[width=8cm]{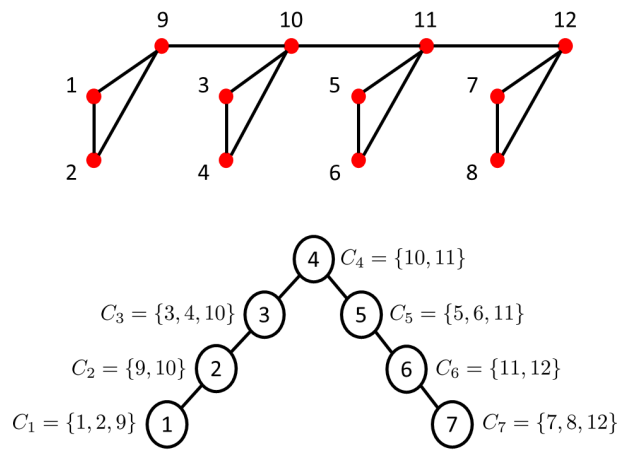}    
\caption{\small A sparsity graph and its corresponding clique tree for the problem in \eqref{eq:example1}.\normalsize }
\label{fig:ex1}
\end{center}
\end{figure}
As was discussed above, the clique tree of the sparsity graph of a coupled problem, defines the computational graph for the distributed algorithm that solves it. Given the sparsity graph for the problem, one of the ways for computing a chordal embedding and a clique tree for this graph is through the use of algorithms~\ref{alg:embed} and \ref{alg:span}. Particularly, using these algorithms allows one to automate the procedure for producing a clique tree for any given sparsity graph, with possibly different outcomes depending on the choice of algorithms. However, it is important to note that sometimes manually adding edges to the sparsity graph or its chordal embedding can enable us to shape the clique tree to our benefit and produce more suitable distributed solutions. In this case, though, extra care must be taken. For instance, it is important to assure that the modified sparsity graph is still a reasonable representation of the coupling in the problem and that the generated tree satisfies the clique intersection property, and is in fact a clique tree, as this property has been essential in the proof of the theorems presented in this section. We illustrate this using an example. Consider the following coupled optimization problem
\begin{subequations}\label{eq:example1}
\begin{align}
\minimize & \quad f_1(x_1,x_2) + f_2(x_3,x_4) + f_3(x_5,x_6) + f_4(x_7,x_8) \\
\subject & \quad g_1(x_1, x_2, x_9) \leq 0\label{eq:example1b}\\
& \quad  g_2(x_3, x_4, x_{10}) \leq 0 \\
& \quad  g_3(x_5, x_6, x_{11}) \leq 0 \\
& \quad  g_4(x_7, x_8, x_{12}) \leq 0 \\
& \quad g_5 (x_{10}, x_{11}) \leq 0\\
& \quad  x_9 - x_{10} = 0\\
& \quad x_{11} - x_{12} = 0.\label{eq:example1h}
\end{align}
\end{subequations}
This problem can be equivalently rewritten as
\begin{multline*}
\minimize \quad f_1(x_1,x_2) + \mathcal I_{\mathcal C_1}(x_1, x_2, x_9) + f_2(x_3,x_4) + \mathcal I_{\mathcal C_2}(x_3, x_4, x_{10}) +\\ f_3(x_5,x_6) + \mathcal I_{\mathcal C_3}(x_5, x_6, x_{11}) + f_4(x_7,x_8) + \mathcal I_{\mathcal C_4}(x_7, x_8, x_{12}) +\\ \mathcal I_{\mathcal C_5}(x_{10}, x_{11}) + \mathcal I_{\mathcal C_6}(x_9, x_{10}) + \mathcal I_{\mathcal C_7}(x_{11}, x_{12}),
\end{multline*}
where $\mathcal I_{\mathcal C_i}$ for $i = 1, \dots, 7$, are the indicator functions for the constraints in \eqref{eq:example1b}-- \eqref{eq:example1h}, respectively, defined as
\begin{align*}
\mathcal I_{\mathcal C_i}(x) = \begin{cases} 0 \hspace{8mm} x \in \mathcal C_i \\ \infty \hspace{6mm} \text{Otherwise} \end{cases}.
\end{align*}
This problem is in the same format as \eqref{eq:CPS}. Let us assume that we intend to produce a distributed algorithm for solving this problem using message-passing that would take full advantage of parallel computations. Without using any intuition regarding the problem and/or incorporating any particular preference regarding the resulting distributed algorithm, we can produce the chordal sparsity graph for this problem as depicted in the top graph of Figure \ref{fig:ex1}. A clique tree for this sparsity graph can be computed using algorithms \ref{alg:embed} and \ref{alg:span}, which is illustrated in the bottom plot of Figure \ref{fig:ex1}. A distributed algorithm based on this computational graph does not take full advantage of parallel computations.
\begin{figure}[t]
\begin{center}
\includegraphics[width=7.5cm]{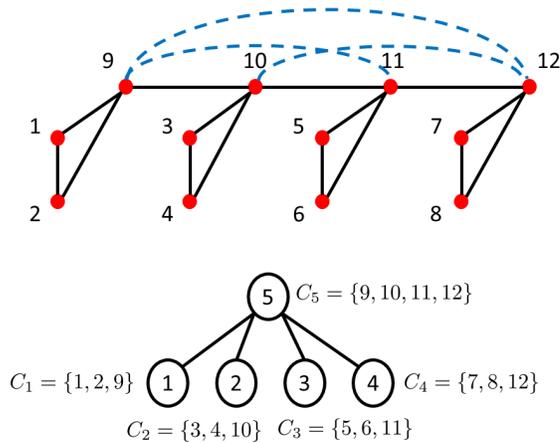}    
\caption{\small An alternative sparsity graph and its corresponding clique tree for the problem in \eqref{eq:example1}.\normalsize }
\label{fig:ex2}
\end{center}
\end{figure}
In order to produce a distributed algorithm that better facilitates the use of parallel computations, it is possible to modify the sparsity graph of the problem as shown in Figure \ref{fig:ex2}, the top graph, where we have added additional edges, marked with dashed lines, to the graph while preserving its chordal property. Notice that by doing so, we have virtually grouped variables $x_9$--$x_{12}$, that couple the terms in the objective function and constraints, together. The corresponding clique tree for this graph is illustrated in Figure \ref{fig:ex2}, the bottom graph. Notice that due to the special structure in the clique tree, within the message-passing algorithm the computation of the messages generated from agents 1--4 can be done independently, and hence in parallel. So using this clique tree as the computational graph of the algorithm enables us to fully take advantage of parallel computations. Next we briefly describe a primal-dual interior-point method for solving convex optimization problems, and then we investigate the possibility of devising distributed algorithms based on these methods for solving loosely coupled problems.

\section{Primal-dual Interior-point Method}\label{sec:PDIPM}
Consider the following convex optimization problem
\begin{equation}\label{eq:ConvexIneq}
\begin{split}
\minimize & \quad  F(x)\\
\subject & \quad  g_i(x) \leq 0, \quad i = 1, \dots, m, \\
&\quad  Ax = b,
\end{split}
\end{equation}
where $F \colon \mathbb R^n \rightarrow \mathbb R$, $g_i \colon \mathbb R^n \rightarrow \mathbb R$ and $A \in \mathbb R^{p \times n}$ with $p<n$ and $\rank (A) = p$. Under the assumption that we have constraint qualification, e.g., that there exist a strictly feasible point, then $x^*$, $v^*$ and $\lambda^*$ constitute a primal-dual optimal solution for \eqref{eq:ConvexIneq} if and only if they satisfy the KKT optimality conditions for this problem, given as
\begin{subequations}\label{eq:ConvexIneqKKTP1}
\begin{align}
\nabla F(x)  + \sum_{i = 1}^m \lambda_i\nabla g_i(x) + A^T v&= 0,\\ \lambda_i&\geq 0, \quad i= 1, \dots, m, \\ g_i(x) &\leq 0, \quad i= 1, \dots, m,\\ -\lambda_i g_i(x) &= 0, \quad i= 1, \dots, m,  \label{eq:ConvexIneqKKTP-d}\\  A x &= b.
\end{align}
\end{subequations}
A primal-dual interior-point method computes such a solution by iteratively solving linearized perturbed versions of \eqref{eq:ConvexIneqKKTP1} where \eqref{eq:ConvexIneqKKTP-d} is modified as
\begin{align*}
-\lambda_i g_i(x) = 1/t, \quad i= 1, \dots, m,
\end{align*}
with $t>0$, \cite{wri:97,boyd:04}. Particularly, for this framework, at each iteration $l$ given primal and dual iterates $x^{(l)}$, $\lambda^{(l)}$ and $v^{(l)}$ so that $g_i(x^{(l)}) < 0$ and $\lambda_i^{(l)} > 0$ for all $i = 1, \dots, m$, the next update direction is computed by solving the linearization of
\begin{subequations}\label{eq:ConvexIneqKKTPD}
\begin{align}
\nabla F(x)  + \sum_{i = 1}^m \lambda_i\nabla g_i(x) + A^T v&= 0,\\  -\lambda_i g_i(x) &= 1/t ,\ i= 1, \dots, m, \label{eq:ConvexIneqKKTPDb}\\  A x &= b.
\end{align}
\end{subequations}
at the current iterates, given as
\begin{subequations}\label{eq:QAppIneqKKTPD}
\begin{align}
\begin{split}
\!\!\!\!\left(\!\nabla^2 F(x^{(l)}) + \sum_{i = 1}^m \lambda_i^{(l)}\nabla^2 g_i(x^{(l)})\right)\! \Delta x+ & \\
 \sum_{i = 1}^m \nabla g_i(x^{(l)})\Delta \lambda_i + A^T\Delta v &= -r^{(l)}_{\text{dual}},\end{split}\\
\begin{split}
 -\lambda_i^{(l)} \nabla g_i(x^{(l)})^T \Delta x - g_i (x^{(l)}) \Delta \lambda_i &= -\left(r^{(l)}_{\text{cent}}\right)_i,\\
  i= 1, \dots, m,&\end{split} \\
 A \Delta x &= -r^{(l)}_{\text{primal}},
\end{align}
\end{subequations}
where
\begin{subequations}\label{eq:res}
\begin{align}
r_{\text{dual}}^{(l)} &= \nabla F(x^{(l)}) +  \sum_{i = 1}^m \lambda^{(l)}_i\nabla g_i(x^{(l)}) + A^T v^{(l)},\\
\left(r_{\text{cent}}^{(l)} \right)_i &= -\lambda^{(l)}_i g_i(x^{(l)})-1/t, \quad i=1, \dots, m,\\
r_{\text{primal}}^{(l)}& = Ax^{(l)} - b.
\end{align}
\end{subequations}
Define $G_d^{(l)} = \diag (g_1(x^{(l)}), \dots, g_1(x^{(l)}))$, $Dg(x) = \begin{bmatrix} \nabla g_1(x) & \dots & \nabla g_m(x) \end{bmatrix}^T$,
\begin{align*}
H_{\text{pd}}^{(l)} = \nabla^2 F(x^{(l)}) + \sum_{i = 1}^m \lambda_i^{(l)}\nabla^2 g_i(x^{(l)})- \sum_{i=1}^m \frac{\lambda_i^{(l)}}{g_i(x^{(l)})} \nabla g_i(x^{(l)})\nabla g_i(x^{(l)})^T,
\end{align*}
and $r^{(l)} = r^{(l)}_{\text{dual}} + Dg(x^{(l)}) ^TG_{\textrm{d}}^{-1} r^{(l)}_{\text{cent}}$. By eliminating $\Delta \lambda$ as
\begin{align}\label{eq:PDLambda}
\Delta \lambda = -G_{\textrm{d}}(x^{(l)})^{-1}\left( \diag(\lambda^{(l)}) Dg(x^{(l)}) \Delta x -   r^{(l)}_{\text{cent}} \right),
\end{align}
we can rewrite \eqref{eq:QAppIneqKKTPD} as
\begin{align}\label{eq:PD}
\begin{bmatrix} H^{(l)}_{\text{pd}} & A^T \\ A & 0 \end{bmatrix}\begin{bmatrix} \Delta x \\ \Delta v \end{bmatrix}  = - \begin{bmatrix}r^{(l)}\\ r^{(l)}_{\text{primal}}     \end{bmatrix},
\end{align}
which has a lower dimension than \eqref{eq:QAppIneqKKTPD}, and unlike the system of equations in \eqref{eq:QAppIneqKKTPD}, is symmetric. This system of equations is sometimes referred to as the augmented system. It is also possible to further eliminate $\Delta x$ in \eqref{eq:PD} and then solve the so-called normal equations for computing $\Delta v$. However, this commonly destroys the inherent structure in the problem, and hence we abstain from performing any further elimination of variables. The system of equations in \eqref{eq:PD} also expresses the optimality conditions for the following quadratic program
\begin{align}\label{eq:PDQP}
\begin{split}
\minimize & \quad \frac{1}{2}\Delta x^T H^{(l)}_{pd} \Delta x + (r^{(l)})^T \Delta x\\
\subject & \quad A \Delta x = - r_{\text{primal}}^{(l)},
\end{split}
\end{align}
and hence, we can compute $\Delta x$ and $\Delta v$ also by solving~\eqref{eq:PDQP}. Having computed $\Delta x$ and $\Delta v$, $\Delta \lambda$ can then be computed using \eqref{eq:PDLambda}, which then allows us to update the iterates along the computed directions. A layout for a primal-dual interior-point is given in Algorithm \ref{alg:PD}.
\begin{algorithm}[tb]
\caption{Primal-dual Interior-point Method, \cite[]{boyd:04}}\label{alg:PD}
\begin{algorithmic}[1]
\small
\State{Given $l = 0$, $\mu>1$, $\epsilon>0$, $\epsilon_{\text{feas}}>0$, $\lambda^{(0)} > 0$, $v^{(0)}$, $x^{(0)}$ such that $g_i(x^{(0)}) < 0$ for all $i = 1, \dots, m$ and $\hat \eta^{(0)} = \sum_{i=1}^m -\lambda_i^{(0)}g_i(x^{(0)})$}
\Repeat
\State{$t = \mu m/\hat \eta^{(l)}$}
\State{Given $t$, $\lambda^{(l)}$, $v^{(l)}$ and $x^{(l)}$ compute $\Delta x^{(l+1)}$, $\Delta \lambda^{(l+1)}$, $\Delta v^{(l+1)}$ by solving \eqref{eq:PD} and~\eqref{eq:PDLambda}}
\State{Compute $\alpha^{(l+1)}$ using line search}
\State  $x^{(l+1)} =  x^{(l)} + \alpha^{(l+1)}\Delta x^{(l+1)}$
\State  $\lambda^{(l+1)} =  \lambda^{(l)} + \alpha^{(l+1)}\Delta \lambda^{(l+1)}$
\State  $v^{(l+1)} =  v^{(l)} + \alpha^{(l+1)}\Delta v^{(l+1)}$
\State {$l = l + 1$}
\State{$\hat \eta^{(l)} = \sum_{i=1}^m -\lambda_i^{(l)}g_i(x^{(l)})$}
\Until{$\| r^{(l)}_{\text{primal}} \|^2, \| r^{(l)}_{\text{dual}} \|^2 \leq \epsilon_{\text{feas}}$ and $\hat \eta^{(l)} \leq \epsilon$}
\normalsize
\end{algorithmic}
\end{algorithm}
\begin{rem}\label{rem:nonsingular}
\emph{Notice that in order for the computed directions to constitute a suitable search direction, the coefficient matrix in \eqref{eq:PD} needs to be nonsingular. There are different assumptions that guarantee such property, e.g., that $\mathcal N(H_{pd}^{(l)}) \cap \mathcal N(A) = \{ 0 \}$, \emph{\cite{boyd:04}}. So, we assume that the problems we consider satisfy this property.}
\end{rem}
There are different approaches for computing proper step sizes in the 5th step of the algorithm. One of such approaches ensures that $g_i(x^{(l+1)})<0$ for $i = 1, \dots, m$ and $\lambda^{(l+1)}\succ 0$, by first setting
\begin{align*}
\alpha_{\textrm{max}} = \minimum \left\{ 1, \minimum \left\{ -\lambda_i^{(l)}/\Delta \lambda_i^{(l+1)}  \ \big | \ \Delta \lambda_i^{(l+1)} < 0  \right\} \right\},
\end{align*}
and conducting a backtracking line search as below\\
\begin{algorithmic}
  \While{$\exists \ i  \colon g_i (x^{(l)} + \alpha^{(l+1)} \Delta x^{(l+1)}) > 0 $}
    \State $\alpha^{(l+1)} = \beta \alpha^{(l+1)}$
  \EndWhile
\end{algorithmic}
with $\beta \in (0,1)$ and $\alpha^{(l+1)}$ initialized as $0.99 \alpha_{\textrm{max}}$. Moreover, in order to ensure steady decrease of the primal and dual residuals, the back tracking is continued as
\begin{flushleft}
\begin{algorithmic}
 \While{$\left\| \left(r_{\text{primal}}^{(l+1)}, r_{\text{dual}}^{(l+1)}\right) \right\| > (1 - \gamma \alpha^{(l+1)})  \left\| \left(r_{\text{primal}}^{(l)}, r_{\text{dual}}^{(l)}\right) \right\| $}
    \State  $\alpha^{(l+1)} = \beta \alpha^{(l+1)}$
  \EndWhile
\end{algorithmic}
\end{flushleft}
where $\gamma \in [0.01, 0.1]$. The resulting $\alpha^{(l+1)}$ ensures that the primal and dual iterates remain feasible at each iteration and that the primal and dual residuals will converge to zero, \cite{wri:97,boyd:04}.
\begin{rem}\label{rem:infeas}
\emph{The primal-dual interior-point method presented in Algorithm \ref{alg:PD}, is an infeasible long step variant of such methods, \emph{\cite{wri:97}}. There are other alternative implementations of primal-dual methods that particularly differ in their choice of search directions, namely short-step, predictor-corrector and Mehrotra's predictor-corrector. The main difference between the distinct primal-dual directions, commonly arise due to different approaches for perturbing the KKT conditions, specially through the choice of $t$, \emph{\cite{wri:97}}. This means that for the linear system of equations in \eqref{eq:ConvexIneqKKTPD}, only the right hand side of the equations will be different and hence the structure of the coefficient matrix in \eqref{eq:PD} remains the same for all the aforementioned variants. Consequently, all the upcoming discussions will be valid for other such variants.}
\end{rem}
Next we provide a formal description of loosely coupled problems and will show how we can devise a distributed primal-dual interior-point method for solving these problems using message-passing.

\section{A Distributed Primal-dual Interior-point Method}\label{sec:DPDIPM}

In this section we put forth a distributed primal-dual interior-point method for solving loosely coupled problems. Particularly, we first provide a formal description for loosely coupled problems and then give details on how to compute the primal-dual directions and proper step sizes, and how to decide on terminating the algorithm distributedly.

\subsection{Loosely Coupled Optimization Problems}\label{sec:loose}

Consider the convex optimization problem in \eqref{eq:EOP}. We can provide mathematical and graphical descriptions of the coupling structure in this problem, as in Section \ref{sec:CPP}. The only difference is that the coupling structure will in this case concern the triplets $f_i, G^i$ and $A^i$ instead of single functions $F_i$. Similar to \eqref{eq:CPS} we can reformulate~\eqref{eq:EOP} as
\begin{subequations}\label{eq:DDEOPS}
\begin{align}
\minimize_{x} \quad & \bar f_1(E_{J_1}x) + \dots + \bar f_N( E_{J_N}x),\\
\subject \quad & \bar G^i(E_{J_i}x) \preceq 0,  \quad  i = 1, \dots, N,\\
  & \bar A^i E_{J_i}x = b^i, \hspace{6mm} i=1, \dots, N, \label{eq:DEOPSc}
\end{align}
\end{subequations}
where in this formulation, the functions $\bar f_i \colon \mathbb R^{|J_i|} \rightarrow \mathbb R$ and $\bar G^i \colon \mathbb R^{|J_i|} \rightarrow \mathbb R^{m_i}$ are defined in the same manner as the functions $\bar F_i$, $\rank(\begin{bmatrix}E_{J_1}^T(\bar A^1)^T & \dots & E_{J_N}^T(\bar A^N)^T \end{bmatrix}^T) = \bar p$ with $\bar p = \sum_{i = 1}^N p_i$, and the matrices $\bar A^i \in \mathbb R^{p_i \times |J_i|}$ are defined by removing unnecessary columns from $A^i$ where $p_i < |J_i|$ and $\rank (\bar A^i) = p_i$ for all $i = 1, \dots, N$. Furthermore, we assume that the loose coupling in the problem is such that the sparsity graph of the problem is such that for all cliques in the clique tree, we have $|C_i| \ll n$ and that $|C_i \cap C_j|$ is small in comparison to the cliques sizes.


From now on let us assume that the chordal sparsity graph of the problem in \eqref{eq:DDEOPS} has $q$ cliques and that $T(V_t, \mathcal E_t)$ defines its corresponding clique tree. Using the guidelines discussed in Section \ref{sec:OMP}, we can then assign different subproblems that build up \eqref{eq:DDEOPS} to each node or agent in the tree. As we will show later, our proposed distributed primal-dual method utilizes this clique tree as its computational graph. Before we go further and in order to make the description of the messages and the message-passing procedure simpler let us group the equality constraints assigned to each agent $j$ as
\begin{align}\label{eq:AgentEq}
\mathcal A^j x = \mathbf b^j
\end{align}
where
\begin{subequations}\label{eq:LocalEquality}
\begin{align}
\mathcal A^j &= \begin{bmatrix} \bar A^{i_1} E_{J_{i_1}} \\ \vdots \\ \bar A^{i_{m_j}} E_{J_{i_{m_j}}}\end{bmatrix},\\
\mathbf b^j &= (b^{i_1}, \dots, b^{i_{m_j}})و
\end{align}
\end{subequations}
for $j = 1, \dots, q$, where $\phi_j = \{ i_1, \dots, i_{m_j}  \}$. We can then rewrite the problem in \eqref{eq:DDEOPS} as
\begin{subequations}\label{eq:DEOPSEq1}
\begin{align}
\minimize \quad &  \bar f_1(E_{J_1}x) + \dots + \bar f_N( E_{J_N}x),\\
\subject \quad & \bar G^i(E_{J_i}x) \preceq 0,  \quad  i = 1, \dots, N,\\
& \mathbf A^i E_{C_i} x = \mathbf b^i, \quad i \in \mathbb N_q \label{eq:DEOPSEq1c}
\end{align}
\end{subequations}
where the coefficient matrices $\mathbf A^i$ are obtained by permuting the columns of the matrices $\mathcal A^i$. Next we solve \eqref{eq:DDEOPS} by applying the primal-dual method in Algorithm \ref{alg:PD} to \eqref{eq:DEOPSEq1} and will discuss how it can be done distributedly within a primal-dual framework. The computational burden of each iteration of a primal-dual interior-point method is dominated by primal-dual directions computation. We hence start by describing a distributed algorithm for calculating these directions using message-passing.

\subsection{Distributed Computation of Primal-dual Directions}\label{sec:DCPDD}

Computing the primal-dual directions requires solving the linear system of equations in~\eqref{eq:PD} where for the problem in \eqref{eq:DEOPSEq1}
\begin{align}\label{eq:QP1}
H_{\text{pd}}^{(l)} = \sum_{i = 1}^{q} \sum_{k \in \phi_i} E_{J_k}^T H_{\text{pd}}^{k,(l)}E_{J_k},
\end{align}
with
\begin{align}\label{eq:QP2}
H_{\text{pd}}^{i,(l)} = \nabla^2 \bar f_i(x_{_{J_i}}^{(l)}) + \sum_{j = 1}^{m_i} \lambda^{i,(l)}_j \nabla^2 \bar G_j^i(x_{_{J_i}}^{(l)})-  \sum_{j=1}^{m_i} \frac{\lambda^{i,(l)}_j}{\bar G_j^i(x_{_{J_i}}^{(l)})}\nabla \bar G_j^i(x_{_{J_i}}^{(l)})\left(\nabla \bar G_j^i(x_{_{J_i}}^{(l)})\right)^T,
\end{align}
$A = \blkdiag\left(\mathbf A^1, \dots, \mathbf A^q\right) \bar E$ with $\bar{E}  = \begin{bmatrix} E_{C_1}^T  &\cdots &  E_{C_q}^T \end{bmatrix}^T$, $r^{(l)} = \bar E^T (r^{1,(l)}, \dots, r^{q,(l)})$ where
\begin{multline*}
r^{i,(l)} =  \sum_{k \in \phi_i}\Big \{\nabla \bar f_k (x_{_{J_k}}^{(l)}) + \sum_{j=1}^{m_k} \lambda^{k,(l)}_j \nabla \bar G_j^k(x_{_{J_k}}^{(l)}) +  \\  D\bar G^k(x_{_{J_k}}^{(l)}) \diag\left(\bar G^k(x_{_{J_k}}^{(l)})\right)^{-1} r_{\text{cent}}^{k,(l)} \Big \} + (\mathbf A^i)^T v^{i,(l)}  ,
\end{multline*}
with
\begin{align*}
r^{k,(l)}_{\text{cent}} = - \diag(\lambda^{k,(l)}) \bar G^k(x_{_{J_k}}^{(l)}) -\frac{1}{t} \mathbf{1},
\end{align*}
and $r_{\text{primal}}^{(l)} = (r_{\text{primal}}^{1,(l)}, \dots, r_{\text{primal}}^{q,(l)})$ with
\begin{align}\label{eq:primalRes}
r_{\text{primal}}^{i,(l)} = & \mathbf A^i x_{_{C_i}}^{(l)} - \mathbf b^i.
\end{align}
The key for devising a distributed algorithm based on a primal-dual interior-point method, is to exploit the structure in this linear system of equations that also expresses the optimality conditions for the following quadratic program
\begin{subequations}\label{eq:QAppEqLogPD}
\begin{align}
\minimize & \quad \sum_{i = 1}^{q} \frac{1}{2}\Delta x^T \left( \sum_{k \in \phi_i} E_{J_k}^T H_{\text{pd}}^{k,(l)}E_{J_k} \right) \Delta x + (r^{i,(l)})^T E_{C_i}\Delta x \\
\subject & \quad \mathbf A^i E_{C_i}(\Delta x + x^{(l)}) = \mathbf b^i, \quad i = 1, \dots, q.\label{eq:QAppEqLogPD-b}
\end{align}
\end{subequations}
which can be rewritten as
\begin{subequations}\label{eq:QAppEqLogPD1}
\begin{align}
\minimize & \quad \sum_{i = 1}^{q} \frac{1}{2}\Delta x^TE_{C_i}^T \mathbf H_{\text{pd}}^{i,(l)} E_{C_i} \Delta x + (r^{i,(l)})^T E_{C_i}\Delta x \\
\subject & \quad \mathbf A^i E_{C_i}(\Delta x + x^{(l)}) = \mathbf b^i, \quad i = 1, \dots, q.\label{eq:QAppEqLogPD-b}
\end{align}
\end{subequations}
where $\mathbf H_{\text{pd}}^{i,(l)} = \sum_{k \in \phi_i} (\bar E^i_{k})^T H_{\text{pd}}^{k,(l)} \bar E^i_{k}$ with $\bar E^i_{k} = E_{J_k}E_{C_i}^T $. In order to assure that the property in Remark \ref{rem:nonsingular} also holds for the problem in \eqref{eq:QAppEqLogPD1}, we need to make assumptions regarding the subproblems assigned to each agent, which is described in the following lemma.
\begin{lemma}\label{lem:lemD}
The condition in Remark \ref{rem:nonsingular} holds for the problem in \eqref{eq:QAppEqLogPD1}, if $\mathcal N(\mathbf H_{\text{pd}}^{i,(l)}) \cap \mathcal N(\mathbf A^i) = \{ 0 \}$ for all subproblems $i \in \mathbb N_q$.
\end{lemma}
\begin{proof}
The condition in Remark \ref{rem:nonsingular} is equivalent to
\begin{align}\label{eq:app12}
\mathcal N\left( \sum_{i=1}^q E_{C_i}^T \mathbf H_{\text{pd}}^{i,(l)} E_{C_i}\right) \cap \mathcal N\left( \begin{bmatrix} \mathbf A^1 E_{C_1} \\ \vdots \\ \mathbf A^q E_{C_q} \end{bmatrix}  \right) = \{ 0 \}.
\end{align}
Since $E_{C_i}^T \mathbf H_{\text{pd}}^{i,(l)} E_{C_i} \in \mathbb S_+^n$ for all $i = 1, \dots, N$, this condition can be equivalently rewritten as
\begin{align}\label{eq:app22}
\left[ \bigcap_{i=1}^q \mathcal N\left(E_{C_i}^T \mathbf H_{\text{pd}}^{i,(l)} E_{C_i}\right)   \right] \cap \left[ \bigcap_{i=1}^q \mathcal N\left(\mathbf A^i E_{C_i}\right)   \right] = \{ 0 \}.
\end{align}
By arranging the terms in \eqref{eq:app22} and using associative property of the intersection operator, we can equivalently reformulate it as
\begin{align}\label{eq:app32}
\bigcap_{i=1}^q \left[  \mathcal N\left(E_{C_i}^T \mathbf H_{\text{pd}}^{i,(l)} E_{C_i}\right)  \cap \mathcal N\left(\mathbf A^i E_{C_i}\right)   \right] = \{ 0 \}.
\end{align}
Notice that the $E_{C_i}$s are constructed such that they have full row rank. Now let $\mathcal N\left(\mathbf H_{\text{pd}}^{i,(l)}\right)  \cap \mathcal N\left(\mathbf A^i\right) = \{ 0 \}$ for all $i = 1, \dots, q$, and assume that there exists $x \neq 0$ such that
\begin{align*}
x \in \bigcap_{i=1}^q \left[  \mathcal N\left(E_{C_i}^T \mathbf H_{\text{pd}}^{i,(l)} E_{C_i}\right)  \cap \mathcal N\left(\mathbf A^i E_{C_i}\right)   \right].
\end{align*}
This then implies that for any $x_{_{C_i}} = E_{C_i}x$ it must hold that $x_{_{C_i}} \in \mathcal N\left(E_{C_i}^T\mathbf H_{\text{pd}}^{i,(l)}\right)  \cap \mathcal N\left(\mathbf A^i\right)$ for all $i = 1, \dots, q$, or equivalently $x_{_{C_i}} \in \mathcal N\left(\mathbf H_{\text{pd}}^{i,(l)}\right)  \cap \mathcal N\left(\mathbf  A^i\right)$ for all $i = 1, \dots, q$, since $E_{C_i}$s have full row rank. Under the assumption that $x \neq 0$, then for some $i\in \mathbb N_q$, $x_{_{C_i}} \neq 0$. Therefore, $x_{C_i} \in \mathcal N\left(\mathbf H_{\text{pd}}^{i,(l)}\right)  \cap \mathcal N\left(\mathbf A^i\right)$ and $x_{_{C_i}}  \neq 0$ for some $i$. This is in contradiction to the assumption that $\mathcal N\left(\mathbf H_{\text{pd}}^{i,(l)}\right)  \cap \mathcal N\left(\mathbf A^i\right) = \{ 0 \}$ for all $i = 1, \dots, q$. This completes the proof.
\end{proof}
We can rewrite \eqref{eq:QAppEqLogPD1} as the following unconstrained optimization problem
\begin{align}\label{eq:CPSQP}
\minimize_{\Delta x} \hspace{4mm} \sum_{i = 1}^{q} \underbrace{\frac{1}{2}\Delta x_{_{C_i}}^T \mathbf H_{\text{pd}}^{i,(l)}\Delta x_{_{C_i}} + (r^{i,(l)})^T \Delta x_{_{C_i}} + \mathcal I_{\mathcal T_i} (\Delta x_{_{C_i}})}_{\bar F_i (\Delta x_{_{C_i}})}
\end{align}
where $\mathcal T_i$ is the polyhedral set defined by the $i$th equality constraint in \eqref{eq:QAppEqLogPD1} and $\mathcal I_{\mathcal T_i}$ is its corresponding indicator function. The problem in \eqref{eq:CPSQP} is in the same form as \eqref{eq:CPS}. Notice that the coupling structure in this problem remains the same during the primal-dual iterations. Furthermore, the coupling structure for this problem is such that we can solve it by performing message-passing over the clique tree for the sparsity graph of~\eqref{eq:DDEOPS}. Considering the subproblem assignments discussed in Section \ref{sec:loose}, at each iteration of the primal-dual method, each agent will have the necessary information to form their corresponding quadratic subproblems and take part in the message passing framework.

Let us now focus on how the exchanged messages can be computed and what information needs to be communicated within the message passing procedure. Firstly, notice that each $\bar F_i$ describes an equality constrained quadratic program. Consequently, computing the exchanged messages for solving~\eqref{eq:CPSQP}, requires us to compute the optimal objective value of equality constrained quadratic programs parametrically as a function of certain variables. We next put forth guidelines on how this can be done efficiently. Consider the following quadratic program
\begin{align}\label{eq:QPxy}
\minimize & \quad  \frac{1}{2} \begin{bmatrix} z \\ y \end{bmatrix}^T \begin{bmatrix} Q_{zz} & Q_{zy}\\ Q_{zy}^T & Q_{yy} \end{bmatrix}\begin{bmatrix} z \\ y \end{bmatrix} + \begin{bmatrix} q_z \\ q_y \end{bmatrix}^T\begin{bmatrix} z \\ y \end{bmatrix} + c\notag\\
\subject &  \quad A_z z + A_y y = \bar b
\end{align}
where $z\in \mathbb R^{n_z}$, $y \in \mathbb R^{n_y}$, $\begin{bmatrix}A_z & A_y\end{bmatrix}\in \mathbb R^{p \times n}$ with $n = n_z + n_y$, $\rank(\begin{bmatrix}A_z & A_y\end{bmatrix}) = \rank(A_z) = p$, and that $\mathcal N(\begin{bmatrix} Q_{zz} & Q_{zy}\\ Q_{zy}^T & Q_{yy} \end{bmatrix}) \cap \mathcal N(\begin{bmatrix}A_z & A_y\end{bmatrix}) = \{0\}$. Without loss of generality assume that we intend to solve this optimization problem parametrically as a function of $y$.  This means that we want to solve the following optimization problem
\begin{align}\label{eq:QPx}
\minimize_z & \quad  \frac{1}{2} z^T Q_{zz} z + z^T (Q_{zy} y + q_z)  + \frac{1}{2} y^T Q_{yy} y  + y^Tq_y + c\notag\\
\subject &  \quad A_z z = \bar b - A_y y
\end{align}
The optimality conditions for this problem are given as
\begin{align}\label{eq:OCQP}
\underbrace{\begin{bmatrix} Q_{zz} & A_{z}^T \\ A_{z} & 0 \end{bmatrix}}_{\mathbf O} \begin{bmatrix} z \\  \bar v \end{bmatrix} = \underbrace{\begin{bmatrix} -q_z \\ \bar b \end{bmatrix} - \begin{bmatrix} Q_{zy} \\ A_{y} \end{bmatrix}y}_{h(y)}
\end{align}
Notice that for the problem in \eqref{eq:QPxy} $\mathbf O$ is nonsingular, which is shown in the following lemma.
\begin{lemma}\label{lem:lemRank}
Consider the problem in \eqref{eq:QPxy}, and assume that $\rank(\begin{bmatrix}A_z & A_y\end{bmatrix}) = \rank(A_z) = p$ and $\mathcal N(\begin{bmatrix} Q_{zz} & Q_{zy}\\ Q_{zy}^T & Q_{yy} \end{bmatrix}) \cap \mathcal N(\begin{bmatrix}A_z & A_y\end{bmatrix}) = \{0\}$. Then $\mathbf O$ is nonsingular.
\end{lemma}
\begin{proof}
Firstly notice that under the assumption in the lemma, the optimality condition for \eqref{eq:QPxy}, given as
\begin{align}
\begin{bmatrix} Q_{zz} & Q_{zy} & A_z^T\\ Q_{zy}^T & Q_{yy} & A_y^T \\ A_z & A_y & 0 \end{bmatrix}\begin{bmatrix} z \\ y \\ \bar v \end{bmatrix} = \begin{bmatrix} -q_z \\ -q_y \\ \bar b \end{bmatrix},
\end{align}
has a unique solution and its coefficient matrix is nonsingular. This means that
\begin{align}
\rank\left( \begin{bmatrix} Q_{zz} \\ Q_{zy}^T \\ A_z \end{bmatrix} \right) = n_z
\end{align}
or equivalently
\begin{align}\label{eq:condlemma1}
\mathcal N\left(\begin{bmatrix}Q_{zz}\\Q_{zy}^T\end{bmatrix}\right) \cap \mathcal N(A_z) = \{ 0 \}.
\end{align}
Since $\begin{bmatrix} Q_{zz} & Q_{zy}\\ Q_{zy}^T & Q_{yy} \end{bmatrix}$ is positive semidefinite, we can rewrite it as
\begin{align}
\begin{bmatrix} Q_{zz} & Q_{zy}\\ Q_{zy}^T & Q_{yy} \end{bmatrix} = \begin{bmatrix} U \\ V \end{bmatrix} \begin{bmatrix} U \\ V \end{bmatrix}^T,
\end{align}
where assuming $\rank\left(\begin{bmatrix} Q_{zz} & Q_{zy}\\ Q_{zy}^T & Q_{yy} \end{bmatrix}\right) = r \leq n$, $\begin{bmatrix} U \\ V \end{bmatrix} \in \mathbb R^{n \times r}$ and has full column rank. Then the condition in \eqref{eq:condlemma1} can be rewritten as
\begin{align}\label{eq:cond1lemma1}
\mathcal N \left(\begin{bmatrix} U \\ V \end{bmatrix}U^T\right) \cap \mathcal N(A_z) = \mathcal N \left(U^T\right) \cap \mathcal N(A_z) = \{ 0 \}.
\end{align}
Furthermore, since $\mathcal C(U^T)$ and $\mathcal N(U)$ are orthogonal complements, we have $\mathcal N \left(UU^T\right) = \mathcal N \left(U^T\right)$, which enables us to rewrite \eqref{eq:cond1lemma1} as
\begin{align}
\mathcal N \left(UU^T\right) \cap \mathcal N(A_z) = \mathcal N \left(Q_{zz}\right) \cap \mathcal N(A_z) = \{ 0 \}
\end{align}
which is equivalent to $\mathbf O$ being nonsingular. This completes the proof.
\end{proof}
By Lemma \ref{lem:lemRank}, we can then solve \eqref{eq:OCQP} as
\begin{equation}\label{eq:Variable}
\begin{split}
 \begin{bmatrix} z\\ \bar  v \end{bmatrix}  &= \mathbf O^{-1}\left(\begin{bmatrix}-q_z \\ \bar b\end{bmatrix}-\begin{bmatrix} Q_{zy} \\ A_y  \end{bmatrix} y\right)\\&= : \begin{bmatrix} H_1 \\ H_2\end{bmatrix}y + \begin{bmatrix} h_1 \\ h_2\end{bmatrix}
\end{split}
\end{equation}
Having computed the optimal solution parametrically as a function of $y$, we can now compute the optimal objective value as a convex quadratic function of $y$, $p^\ast(y)$, by simply substituting $z$ from \eqref{eq:Variable} in the objective function of \eqref{eq:QPxy}.

We can now discuss the computation and content of the messages. Firstly notice that each of the constraints in \eqref{eq:QAppEqLogPD-b} can be written as
\begin{align*}
\mathbf A^i_1 \Delta x_{_{C_i\setminus S_{i\parent(i)}}} + \mathbf A^i_2 \Delta x_{_{S_{i\parent(i)}}} = \mathbf b^i - \mathbf A^i x^{(l)}_{_{C_i}}.
\end{align*}
For now assume that $\begin{bmatrix} \mathbf A^i_1 & \mathbf A^i_2 \end{bmatrix}$ and $\mathbf A^i_1$ are full row rank for all $i \in \mathbb N_q$. Also recall that for the problem in \eqref{eq:CPSQP} the message to be sent from agent $i$ to its parent $\parent(i)$ is given as
\begin{align}\label{eq:mijQ}
m_{ij}(\Delta x_{_{S_{i\parent(i)}}}) = \minimum_{\Delta x_{_{C_i \setminus S_{i\parent(i)}}}} \left\{  \bar F_i(\Delta x_{_{C_i}})  + \sum_{k \in \children(i)} m_{ki}(\Delta x_{_{S_{ik}}}) \right\}.
\end{align}
Then, for this problem, all the exchanged messages define quadratic functions as described above, which is shown in the following theorem.
\begin{theorem}\label{thm:thm63}
Consider the message description given in \eqref{eq:mijQ}. For the problem in \eqref{eq:CPSQP}, all the exchanged messages are quadratic functions.
\end{theorem}
\begin{proof}
We prove this using induction, where we start with the agents at the leaves. For every agent $i \in \leaves(T)$, the computed message to be sent to the corresponding parent can be computed by solving
\begin{subequations}\label{eq:thmQP}
\begin{align}
\minimize & \quad \frac{1}{2} \Delta x_{_{C_i}}^T \mathbf H_{\text{pd}}^{i,(l)} \Delta x_{_{C_i}} + (r^{i,(l)})^T \Delta x_{_{C_i}} \\
\subject & \quad \mathbf A^i (\Delta x_{_{C_i}} + x_{_{C_i}}^{(l)}) = \mathbf b^i,
\end{align}
\end{subequations}
parametrically as a function of $\Delta x_{_{S_{\parent(i)i}}}$. Under the assumption stated in Lemma \ref{lem:lemD}, $\mathcal N(\mathbf H_{\text{pd}}^{i,(l)}) \cap \mathcal N(\mathbf A^i) = \{ 0 \}$. As a result the assumption in Lemma \ref{lem:lemRank} holds for \eqref{eq:thmQP} and hence we can use the procedure discussed above to solve the problem parametrically. Consequently the messages sent from the leaves are quadratic functions. Now consider an agent $i$ in the middle of the tree and assume that all the messages received by this agent are quadratic functions of the form
\begin{align*}
m_{ki}(\Delta x_{_{S_{ik}}}) = \Delta x_{_{S_{ik}}}^T Q_{ki} \Delta x_{_{S_{ik}}} + q_{ki}^T \Delta x_{_{S_{ik}}} + c_{ki}.
\end{align*}
Then this agent can compute the message to be sent to its parent, by solving
\begin{subequations}\label{eq:QPmij}
\begin{align}
\minimize &\quad \frac{1}{2}\Delta x_{_{C_i}}^T \left( \mathbf H_{\text{pd}}^{i,(l)} + \sum_{k \in \children(i)} \bar E_{ik}^TQ_{ki} \bar E_{ik}  \right) \Delta x_{_{C_i}}  \notag\\& \quad  + \left( r^{i,(l)} + \sum_{k \in \children(i)} \bar E_{ik}^T q_{ki} \right)^T\Delta x_{_{C_i}} + \bar c_i\\
\subject &  \quad \mathbf A^i (\Delta x_{_{C_i}} + x_{_{C_i}}^{(l)}) = \mathbf b^i \\
\end{align}
\end{subequations}
with $\bar E_{ik} = E_{S_{ik}}E_{C_i}^T $, parametrically as a function of $\Delta x_{_{S_{i\parent(i)}}}$. Notice that the assumption in Lemma \ref{lem:lemD} implies that $\mathcal N\left(\mathbf H_{\text{pd}}^{i,(l)} + \sum_{k \in \children(i)} \bar E_{ik}^TQ_{ki} \bar E_{ik}\right) \cap \mathcal N(\mathbf A^i) = \{ 0 \}$. This means that the assumption in Lemma \ref{lem:lemRank} would also be satisfied and hence the computed message to the parent would be a quadratic function. This completes the proof.
\end{proof}
Notice that sending the message $m_{ij}$ to agent $j$ requires agent $i$ to send the data matrices that define the quadratic function. Following the steps of the message-passing method discussed in Section~\ref{sec:OMP}, we can now compute the primal variables direction, $\Delta x$, distributedly.

It now remains to discuss how to compute the dual variables directions, $\Delta  v^k$ for $k = 1, \dots, q$, and $\Delta \lambda^k$ for $k = 1, \dots, N$. We will next show that in fact it is possible to compute the optimal dual variables direction during the downward pass of the message-passing algorithm. Firstly recall that during the upward pass each agent $i$, except the agent at the root, having received all the messages from its children forms \eqref{eq:QPmij} and solves it parameterically as a function of $\Delta x_{_{S_{i\parent(i)}}}$ by first computing
\begin{align}\label{eq:dualvariable}
\begin{bmatrix}\Delta x_{_{C_i \setminus S_{i\parent(i)}}}\\\Delta v^i\end{bmatrix} = \begin{bmatrix} H_1^i \\ H_2^i  \end{bmatrix}\Delta x_{_{S_{i\parent(i)}}} + \begin{bmatrix} h_1^i \\ h_2^i \end{bmatrix},
\end{align}
as described above, and then communicating the parametric optimal objective value as the message to the parent. Notice that \eqref{eq:dualvariable} defines the optimality conditions for \eqref{eq:QPmij} given $\Delta x_{_{S_{i\parent(i)}}}$ or equivalently the optimality conditions of \eqref{eq:LocalProblempar} without the regularization term, which we are allowed to neglect since by the assumption in Lemma \ref{lem:lemD} the optimal solution of \eqref{eq:QAppEqLogPD1} is unique, see Remark \ref{rem:rem1}. As a result this agent having received $\Delta x_{_{S_{i\parent(i)}}}^*$ from its parent can use \eqref{eq:dualvariable} to compute its optimal primal solution. As we will show later the computed dual variables in \eqref{eq:dualvariable} during this process will also be optimal for \eqref{eq:QAppEqLogPD1}. The agent at the root can also compute its optimal dual variables in a similar manner. Particularly , this agent having received all the messages from its children can also form the problem in \eqref{eq:QPmij}. Notice that since $\parent(r) = \emptyset$ then $S_{\parent(r)r} = \emptyset$. As a result~\eqref{eq:dualvariable} for this agent becomes
\begin{align}
\begin{bmatrix}\Delta x^\ast_{_{C_r}}\\(\Delta v^r)^*\end{bmatrix} = \begin{bmatrix} h_1^r \\ h_2^r \end{bmatrix},
\end{align}
which is the optimality condition for \eqref{eq:QPmij}. Consequently, the dual variables computed by this agent when calculating its optimal primal variables will in fact be optimal for the problem in \eqref{eq:QAppEqLogPD1}. The next theorem shows that the computed primal directions $\Delta x^*$ and dual directions $\Delta v^*$ using this approach then satisfy the optimality conditions for the complete problem in~\eqref{eq:QAppEqLogPD} and hence constitute a valid update direction, i.e., we can choose $\Delta x^{(l+1)} = \Delta x^*$ and $\Delta v^{(l+1)} = \Delta v^*$.
\begin{theorem}\label{thm:thm5}
If each agent $i \in \mathbb N_q$ computes its corresponding optimal primal and dual variables directions, $\Delta x^\ast_{_{C_i}}, (\Delta v^i)^*$, using the procedure discussed above, then the calculated directions by all agents constitute an optimal primal-dual solution for the problem in \eqref{eq:QAppEqLogPD}.
\end{theorem}
\begin{proof}
We prove this theorem by establishing the connection between the message-passing procedure and row and column manipulations on the KKT optimality conditions of \eqref{eq:QAppEqLogPD1}. So let us start from the leaves of the tree. From the point of view of agents $i \in \leaves(T) = \{ i_1, \dots, i_t \}$ we can rewrite \eqref{eq:QAppEqLogPD1} as
\newpage
\small
\begin{subequations}\label{eq:Agenti}
\begin{align}
\minimize_{y^{i_1}, \dots, y^{i_t},u,z} & \quad \frac{1}{2}\left( \sum_{i \in \leaves(T)} \begin{bmatrix} y^i \\ u \end{bmatrix}^T \begin{bmatrix} R^i_{yy} & \bar R^i_{yu} \\ (\bar R_{yu}^i)^T &  \bar R^i_{uu} \end{bmatrix} \begin{bmatrix} y^i \\ u \end{bmatrix} +  \right. \notag\\ & \left.\hspace{20mm} \begin{bmatrix}q^i_y \\  \bar q^i_u \end{bmatrix}^T \begin{bmatrix} y^i \\ u \end{bmatrix} + c^i\right)  + \frac{1}{2} \begin{bmatrix} u \\ z \end{bmatrix}^T \begin{bmatrix} R_{uu} & R_{uz}\\ R_{uz}^T & R_{zz} \end{bmatrix} \begin{bmatrix}  u \\ z \end{bmatrix} + \begin{bmatrix}q_u \\  q_z \end{bmatrix}^T \begin{bmatrix} u \\ z \end{bmatrix} \label{eq:Agenti-a}  \\
\subject & \quad A_y^i y^i + \bar A_{yu}^i u = b^i_y, \quad i \in \leaves(T) \label{eq:Agenti-b}\\
& \quad A_{zu} u + A_{z} z = b_z,
\end{align}
\end{subequations}
\normalsize
where $y^i = \Delta x_{_{C_i\setminus S_{i\parent(i)}}}$, $u = \Delta x_{_{S_T}}$ with $S_T = \cup _{i \in \leaves(T)} S_{i\parent(i)}$, $z = \Delta x_{_{S_{p} \setminus S_T}}$ with $S_p = \cup_{i\in \leaves(T)} V_{\parent(i)i}$, $\bar R^i_{yu} = R^i_{yu}\bar E_{i\parent(i)}$, $\bar R^i_{uu} = \bar E_{i\parent(i)}^TR^i_{uu}\bar E_{i\parent(i)}$, $\bar q_u^i = \bar E_{i\parent(i)}^Tq^i_u$ and $\bar A_{yu}^i = A_{yu}^i \bar E_{i\parent(i)}$ with $\bar E_{i\parent(i)} = E_{S_{i\parent(i)}}E_{C_i}^T$. In other words, in this formulation the variables $y$, $u$ and $z$ denote the variables present only in the subproblems assigned to the leaves, the variables that appear in both these subproblems and subproblems assigned to all the other agents and the variables that are not present in the subproblems assigned to the leaves, respectively. Furthermore, each of the terms in the sum in \eqref{eq:Agenti-a} and the constraints in \eqref{eq:Agenti-b} denote the cost functions and equality constraints that are assigned to the $i$th agent. The KKT optimality conditions for this problem can be written as

\small
\begin{align}\label{eq:thmKKT}
&\begin{bmatrix}\begin{array}{cccccc:ccccc} R_{yy}^{i_1} & 0 & \dots & 0 & \bar R_{yu}^{i_1} & 0 & (A_{y}^{i_1})^T & 0 & \dots & 0 & 0 \\ 0 & R_{yy}^{i_t} & \dots & 0 &\bar R_{yu}^{i_2} & 0 &  0 & (A_{y}^{i_1})^T & \dots & 0 & 0 \\ \vdots & \vdots & \ddots & \vdots & \vdots & \vdots &\vdots &  \vdots  &\ddots &\vdots & \vdots\\ 0 & 0 & \dots & R_{yy}^{i_t} &\bar R_{yu}^{i_t} & 0 & 0 & 0 & \dots & (A_{y}^{i_t})^T & 0 \\ (\bar R_{yu}^{i_1})^T & (\bar R_{yu}^{i_2})^T & \dots & (\bar R_{yu}^{i_t})^T & \ \ \left(R_{uu} + \sum_{i\in \leaves(T)} \bar R_{uu}^i\right) \ \ & R_{uz} & (\bar A_{yu}^{i_1})^T & (\bar A_{yu}^{i_2})^T & \dots & (\bar A_{yu}^{i_t})^T &  A_{zu}^T \\ 0 & 0 & \dots & 0 & R_{uz}^T & R_{zz} & 0 & 0 & \dots & 0 & A_z^T \\ \hdashline A_{y}^{i_1} & 0 & \dots & 0 & \bar A_{yu}^{i_1} & 0 & 0 & 0 & \dots & 0 & 0 \\ 0 & A_{y}^{i_2} & \dots & 0 & \bar A_{yu}^{i_2} & 0 &  0 & 0 & \dots & 0 & 0 \\ \vdots & \vdots & \ddots & \vdots & \vdots & \vdots &\vdots &  \vdots  &\ddots &\vdots & \vdots \\ 0 & 0 & \dots & A_{y}^{i_t} & \bar A_{yu}^{i_t} & 0 &  0 & 0 & \dots & 0 & 0 \\ 0 & 0 & \dots & 0 &  A_{zu} & A_z &  0 & 0 & \dots & 0 & 0 \end{array} \end{bmatrix} \times \notag \\ & \hspace{70mm} \begin{bmatrix} y^{i_1} \\ y^{i_2} \\ \vdots \\y^{i_t} \\ u \\ z \\ \hdashline \Delta  v^{i_1} \\ \Delta  v^{i_2} \\ \vdots \\ \Delta  v^{i_t}\\ \Delta  v_z\end{bmatrix} = \begin{bmatrix} -q^{i_1}_y \\ -q^{i_2}_y \\ \vdots \\ -q^{i_t}_y\\ - \left(\sum_{i\in \leaves(T)}  \bar q^i_u + q_u\right) \\ -q_z \\ \hdashline b^{i_1}_y \\ b^{i_2}_y \\ \vdots \\ b^{i_t}_y\\ b_z\end{bmatrix},
\end{align}
\normalsize
which by conducting column and row permutations can be rewritten as
\small
\begin{align}\label{eq:thmKKTPer}
&\begin{bmatrix}\begin{array}{cc:cc:c:cc:c:cc} R_{yy}^{i_1} & (A_{y}^{i_1})^T & 0 & 0 & \dots & 0 & 0 & \bar R_{yu}^{i_1} & 0 & 0 \\ A_{y}^{i_1} & 0 & 0 & 0 & \dots & 0 & 0 & \bar A_{yu}^{i_1} & 0 & 0\\\hdashline 0 & 0 & R_{yy}^{i_2} & (A_{y}^{i_2})^T & \dots & 0 & 0 & \bar R_{yu}^{i_2} & 0 & 0\\ 0 & 0 & A_{y}^{i_1} & 0 & \dots & 0 & 0 & \bar A_{yu}^{i_2} & 0 & 0\\\hdashline \vdots & \vdots & \vdots & \vdots & \ddots & \vdots & \vdots & \vdots & \vdots \\\hdashline 0 & 0 & 0 & 0 & \dots & R_{yy}^{i_t} & (A_{y}^{i_t})^T & \bar R_{yu}^{i_t} & 0 & 0\\ 0 & 0 & 0 & 0 & \dots & A_{y}^{i_t} & 0 & \bar A_{yu}^{i_t} & 0 & 0\\\hdashline (\bar R_{yu}^{i_1})^T & (\bar A_{yu}^{i_1})^T & (\bar R_{yu}^{i_2})^T & (\bar A_{yu}^{i_2})^T & \dots & (\bar R_{yu}^{i_t})^T & (\bar A_{yu}^{i_t})^T & \ \ \left(R_{uu} + \sum_{i\in \leaves(T)} \bar R_{uu}^i\right) \ \ & R_{uz} & A_{zu}^T \\\hdashline 0 & 0 & 0 & 0 & \dots & 0 & 0 & R_{uz}^T & R_{zz} & A_z^T \\ 0 & 0 & 0 & 0 & \dots & 0 & 0 & A_{zu} & A_{z} & 0 \end{array} \end{bmatrix} \times \notag \\ & \hspace{70mm} \begin{bmatrix} y^{i_1} \\ \Delta  v^{i_1}  \\\hdashline y^{i_2} \\ \Delta  v^{i_2} \\ \hdashline\vdots \\\hdashline y^{i_t} \\ \Delta  v^{i_t} \\\hdashline  u \\\hdashline z \\ \Delta v_z \end{bmatrix} = \begin{bmatrix} -q^{i_1}_y \\  b^{i_1}_y\\\hdashline -q^{i_2}_y \\ b^{i_2}_y \\\hdashline \vdots \\\hdashline -q^{i_t}_y\\ b^{i_t}_y \\\hdashline - \left(\sum_{i\in \leaves(T)}  \bar q^i_u + q_u\right) \\\hdashline -q_z  \\ b_z\end{bmatrix}.
\end{align}
\normalsize
By Lemma \ref{lem:lemRank}, the blocks $\begin{bmatrix} R_{yy}^{i_j} & (A_{y}^{i_j})^T \\ A_{y}^{i_j} & 0 \end{bmatrix}$  are all nonsingular and hence we can define
\footnotesize
\begin{align}
Q_1 = \begin{bmatrix} \begin{array}{c:c:c:c:c:c}  I & 0 & \dots & 0 & 0 & 0 \\ \hdashline 0 & I & \dots & 0 & 0 & 0\\ \hdashline \vdots & \vdots & \ddots & \vdots & \vdots & \vdots\\ \hdashline 0 & 0 & \dots & I & 0 & 0 \\ \hdashline  -\begin{bmatrix} (\bar R_{yu}^{i_1})^T & (\bar A_{yu}^{i_1})^T \end{bmatrix} (\mathbf O^{i_1})^{-1} & - \begin{bmatrix} (\bar R_{yu}^{i_2})^T & (\bar A_{yu}^{i_2})^T  \end{bmatrix}^{^{^{}} } (\mathbf O^{i_2})^{-1} & \dots & -\begin{bmatrix} (\bar R_{yu}^{i_t})^T & (\bar A_{yu}^{i_t})^T  \end{bmatrix} (\mathbf O^{i_t})^{-1} & I & 0 \\ \hdashline 0 & 0 & 0 & 0 & 0 & I \end{array} \end{bmatrix}
\end{align}
\normalsize
with $\mathbf O^{i_j} = \begin{bmatrix} R^{i_j}_{yy} & (A^{i_j}_y)^T \\A^{i_j}_y & 0 \end{bmatrix}$. If we pre-multiply \eqref{eq:thmKKTPer} by $Q_1$, we can rewrite it as
\small
\begin{subequations}\label{eq:thmitopar}
\begin{align}
&\begin{bmatrix} y^i \\  \Delta v^i  \end{bmatrix} = (\mathbf O^i)^{-1}\left(-\begin{bmatrix} R^i_{yu} \\ A^i_{yu} \end{bmatrix}\bar E_{i\parent(i)} u + \begin{bmatrix} -q^i_y \\ b^i_y \end{bmatrix}\right) =: \begin{bmatrix} H_1^i \\ H_2^i  \end{bmatrix}\bar E_{i\parent(i)} u + \begin{bmatrix} h_1^i \\ h_2^i \end{bmatrix}, \quad i \in \leaves(T)\label{eq:thmitopar-a} \\
& \begin{bmatrix} \left( R_{uu} + \mathbf R_{uu}\right) & R_{uz} & A_{zu}^T \\ R_{uz}^T & R_{zz} & A_z^T \\ A_{zu} & A_z & 0 \end{bmatrix} \begin{bmatrix}  u \\ z \\\Delta v_z \end{bmatrix} = \begin{bmatrix} -(q_u + \mathbf q_u) \\ -q_z \\ b_z \end{bmatrix}\label{eq:thmitopar-c}
\end{align}
\end{subequations}
\normalsize
where
\begin{subequations}\label{eq:thmitopar-b}
\begin{align}
\mathbf R_{uu} &=  \sum_{i\in \leaves(T)}\bar E_{i\parent(i)}^T \left( R_{uu}^i + (H^i_1)^TR^i_{yy} H^i_1 + (H^i_1)^TR^i_{yu} + (R^i_{yu})^T H^i_1 \right)\bar E_{i\parent(i)}\\
\mathbf q_u &= \sum_{i\in \leaves(T)} \bar E_{i\parent(i)}^T \left( q_{u}^i + (H^i_1)^Tq^i_y + (R^i_{yu})^T h^i_1 +(H^i_1)^TR^i_{yy} h^i_1 \right)
\end{align}
\end{subequations}
Notice that considering the definitions in \eqref{eq:Agenti} and \eqref{eq:thmitopar}, the matrices $H^i_1, H^i_2, h^i_1$ and $h^i_2$ in \eqref{eq:thmitopar-a} and \eqref{eq:dualvariable} are the same, and hence the terms $(H^i_1)^TR^i_{yy} H^i_1 + (H^i_1)^TR^i_{yu} + (R^i_{yu})^T H^i_1  + R^i_{uu}$ and $q^i_u + (H^i_1)^Tq^i_y + (R^i_{yu})^T h^i_1 +(H^i_1)^TR^i_{yy} h^i_1$ in \eqref{eq:thmitopar-b} are the data matrices that define the quadratic and linear terms of the message sent from each of the leaves to its parent, and the additional terms $\bar E_{i\parent(i)}^T$ and $\bar E_{i\parent(i)}$ merely assure that the messages are communicated to the corresponding parents. By performing the pre-multiplication above we have in fact pruned the leaves of the tree, and have eliminated the variables that are only present in their respective subproblems. We can now conduct the same procedure outlined in \eqref{eq:thmKKT}--\eqref{eq:thmitopar-b}, that is repartitioning of variables and performing row and column permutations, for the parents that all their children have been pruned, using~\eqref{eq:thmitopar-c}. We continue this approach until we have pruned all the nodes in the tree except for the root, as

\small
\begin{subequations}\label{eq:thmitoR}
\begin{align}
&\begin{bmatrix} \Delta x_{_{C_i\setminus S_{i\parent(i)}}} \\  \Delta v^i  \end{bmatrix} = \begin{bmatrix} H_1^i \\ H_2^i  \end{bmatrix}\Delta x_{_{S_{i\parent(i)}}} + \begin{bmatrix} h_1^i \\ h_2^i \end{bmatrix}, \quad i \in \mathbb N_q \setminus \{ r \}\label{eq:thmitoR-a} \\
& \begin{bmatrix} \left( \mathbf H_{\text{pd}}^{r,(l)} + \sum_{k \in \children(r)}\bar E_{rk}^TQ_{kr} \bar E_{rk}\right) & (\mathbf A^r)^T \\ \mathbf A^r & 0 \end{bmatrix} \begin{bmatrix}  \Delta x_{_{C_r}} \\ \Delta  v^r \end{bmatrix} = \begin{bmatrix} -\left(r^{r,(l)} + \sum_{k \in \children(r)} \bar E_{rk}^T q_{kr}\right) \\  r_{\text{primal}}^{r,(l)}\end{bmatrix}\label{eq:thmitoR-b}
\end{align}
\end{subequations}
\normalsize
where what remains to solve is the optimality conditions for the problem of the agent at the root, given in \eqref{eq:QPmij}, in \eqref{eq:thmitoR-b}. Notice that this procedure is in fact the same as the upward pass in Algorithm \ref{alg:MP}. At this point we can solve \eqref{eq:thmitoR-b} and back substitute the solution in the equations in \eqref{eq:thmitoR-a} with the reverse ordering of the upward pass, which corresponds to the downward pass through the clique tree in Algorithm \ref{alg:MP}. With this we have shown the equivalence between applying the message-passing algorithm to \eqref{eq:QAppEqLogPD1} and solving the KKT conditions of this problem by performing row/column manipulations, and hence have completed the proof.

\end{proof}
Finally, during the downward pass and by \eqref{eq:PDLambda}, each agent having computed its primal variables direction $\Delta x_{_{C_i}}^*$, can compute the dual variables directions corresponding to its inequality constraints by
\begin{align}\label{eq:DualDirectionIn}
\Delta \lambda^{k,(l+1)} = -\diag( \bar G^k(x_{_{J_k}}^{(l)}))^{-1} \left( \diag(\lambda^{k,(l)}) D\bar G^k(x_{_{J_k}}^{(l)}) \Delta x_{_{J_k}}^* -   r^{k,(l)}_{\text{cent}} \right),
\end{align}
for all $k \in \phi_i$.
\begin{rem}
\emph{Notice that the proposed message-passing algorithm for computing the primal-dual directions relies on the assumption that $\mathcal N\left(\mathbf H_{\text{pd}}^{i,(l)} + \sum_{k \in \children(i)} \bar E_{ik}^TQ_{ki} \bar E_{ik}\right) \cap \mathcal N(\mathbf A^i) = \{ 0 \}$ for all $i \in \mathbb N_q$, and the conditions in Lemma \ref{lem:lemD} describe a sufficient condition for this assumption to hold. However, the aforementioned assumption can still hold even if the conditions in Lemma \ref{lem:lemD} are not satisfied, in which case the proposed algorithm can still be used.}
\end{rem}
\begin{rem}
\emph{It is also possible to use a feasible primal interior-point method for solving the problem in \eqref{eq:DDEOPS}. For a primal interior-point method, unlike a primal-dual one, at first the KKT optimality conditions are equivalently modified by eliminating the dual variables corresponding to the inequality constraints, using the perturbed complementarity conditions as in \eqref{eq:ConvexIneqKKTPDb}. Then the resulting nonlinear system of equations is solved using the Newton method, iteratively, \emph{\cite[11.3.4]{boyd:04}}. At each iteration of a feasible primal interior-point method, we only need to update the primal variables, where their corresponding update direction is computed by solving a linear system of equations similar to \eqref{eq:PD}. In fact, applying a primal interior-point method to the problem in \eqref{eq:DDEOPS} would then, at each iteration, require solving a linear system of equations that will have the same structure as the one we solve in a primal-dual interior-point method. Hence, we can use the same message-passing procedure discussed above to compute the primal variables directions within a primal framework. Primal interior-point methods are known to perform worse than their primal-dual counterparts. However, since we do not need to compute dual variables directions at each iteration, we can relax the rank condition on the equality constraints. This is because this condition has solely been used for the proof of Theorem \ref{thm:thm5}, and only concerns the computations of the dual variables.}
\end{rem}

The distributed algorithm for computing the primal-dual directions in this section relies on the seemingly restrictive rank conditions that  $\begin{bmatrix}E_{J_1}^T(\bar A^1)^T & \dots & E_{J_N}^T(\bar A^N)^T \end{bmatrix}^T$, $\begin{bmatrix} \mathbf A^i_1 & \mathbf A^i_2 \end{bmatrix}$ and $\mathbf A^i_1$ are all full row rank for all $i \in \mathbb N_q$. Next we show that these conditions do not affect the generality of the algorithm and in fact they can be imposed by conducting a preprocessing of equality constraints.

\subsubsection{Preprocessing of the Equality Constraints}\label{sec:preprocess}


We can impose the necessary rank conditions by conducting a preprocessing on the equality constraints, prior to application of the primal-dual method. This preprocessing can be conducted distributedly over the same tree used for computing the search directions. Let us assume that the constraints assigned to each of the agents at the leaves, i.e., all $i \in \leaves(T)$, are given as
\begin{align}\label{eq:equalityConsti}
\bar{\mathbf A}^i_1 x_{_{C_i\setminus S_{i\parent(i)}}} + \bar{\mathbf A}^i_2 x_{_{S_{i\parent(i)}}} = \bar{\mathbf b}^i,
\end{align}
and that $\begin{bmatrix}\bar{\mathbf A}^i_1 & \bar{\mathbf A}^i_2 \end{bmatrix} \in \mathbb R^{\bar p_i \times n_i}$ and that $\rank(\bar{\mathbf A}^i_1) = q_i < \bar p_i$. Every such agent can then compute a rank revealing QR factorization for $\bar{\mathbf A}^i_1$ as
\begin{align}
\bar{\mathbf A}^i_1 = Q^i \begin{bmatrix} R^i \\ 0\end{bmatrix},
\end{align}
where $Q^i \in \mathbb R^{\bar p_i \times \bar p_i}$ is an orthonormal matrix and $R^i \in \mathbb R^{q_i \times |C_i \setminus S_{i\parent(i)}|}$ with $\rank(R^i) = q_i$. As a result the constraints in \eqref{eq:equalityConsti} can be equivalently rewritten as
\begin{align}\label{eq:equalityConstiReform}
\begin{bmatrix} \mathbf A_1^i & \mathbf A^i_2 \\ 0 & \mathbf A^i_3 \end{bmatrix} x_{_{C_i}} = \begin{bmatrix} \mathbf b^i\\ \hat{\mathbf b}^i \end{bmatrix}
\end{align}
where
\begin{align*}
\begin{bmatrix} \mathbf A_1^i & \mathbf A^i_2 \\ 0 & \mathbf A^i_3 \end{bmatrix} &:=Q^i \begin{bmatrix} \bar{\mathbf A}^i_1 & \bar{\mathbf A}^i_2 \end{bmatrix}\\
\begin{bmatrix} \mathbf b^i\\  \hat{\mathbf b}^i \end{bmatrix} & := Q^i \bar{\mathbf b}^i.
\end{align*}
Once each agent at the leaves has computed the reformulation of its equality constraints, it will then remove the equality constraints defined by the second row equations in \eqref{eq:equalityConstiReform} from its equality constraints, and communicates them to its parent. At this point the equality constraints assigned to each agent $i$ at the leaves, becomes 
\begin{align*}
\begin{bmatrix} \mathbf A_1^i & \mathbf A^i_2 \end{bmatrix} x_{_{C_i}} = \mathbf b^i,
\end{align*}
where $\begin{bmatrix} \mathbf A_1^i & \mathbf A^i_2 \end{bmatrix}$ and $\mathbf A_1^i$ are both full row rank.
Then every parent that has received all the equality constraints from its children, appends these constraints to its own set of equality constraints, and performs the same procedure as was conducted by the agents at the leaves. This process is then continued until we reach the root of the tree. The agent at the root will then conduct the same reformulation of its corresponding equality constraints and removes the unnecessary trivial equality constraints. Notice that at this point the equality constraints for all agents satisfy the necessary rank conditions, and hence the preprocessing is accomplished after an upward pass through the tree.
\begin{remark}
In a similar manner as in the proof of Theorem \ref{thm:thm5}, it can be shown that the preprocessing procedure presented in this section (except for the removal of trivial constraints by the agent at the root) can be viewed as conducting column permutations on the coefficient matrix of the equality constraints and pre-multiplying it by a nonsingular matrix. Consequently, this preprocessing of the equality constraints, does not change the feasible set.
\end{remark}

In the proof of Theorem \ref{thm:thm5} we described the equivalence between applying the message-passing scheme in Algorithm \ref{alg:MP} to the problem in \eqref{eq:QAppEqLogPD1} and solving its corresponding KKT system through column and row manipulations. Inspired by this discussion, and before we describe distributed implementations of other components of the primal-dual interior-point method in Algorithm \ref{alg:PD}, we explore how the message-passing algorithm in \ref{alg:MP} can be construed as a multi-frontal factorization technique.

\subsection{Relations to Multi-frontal Factorization Techniques}\label{sec:fact}

Let us compactly rewrite the KKT system in \eqref{eq:thmKKT} as
\begin{align}
H \begin{bmatrix} y \\ u\\ z \\ \Delta v  \end{bmatrix} = r.
\end{align}
Then \eqref{eq:thmKKTPer} can be written as
\begin{align}\label{eq:compactKKTPer}
P_1HP_1^T P_1 \begin{bmatrix} y \\ u\\ z \\ \Delta v  \end{bmatrix} = P_1 r,
\end{align}
where $P_1$ is a permutation matrix. In the proof of Theorem \ref{thm:thm5} we showed that by pre-multiplying \eqref{eq:thmKKTPer} by $Q_1$, we can block upper-triangulate the KKT system as in \eqref{eq:thmitopar}, i.e., $Q_1P_1HP_1^T$ is block upper-triangular. This was in fact equivalent to the first stage of the upward pass in Algorithm \ref{alg:MP}, which corresponds to sending messages from the agents at the leaves of the tree to their parents. If we now in this stage multiply $Q_1P_1HP_1^T$ from the right by $Q_1^T$, it is straightforward to verify that we arrive at
\begin{align}\label{eq:factorKKTDiag}
Q_1P_1HP_1^TQ_1^T = \begin{bmatrix}\begin{array}{c:c:c:c:c} \blacksquare & 0 & \dots & 0 & 0 \\\hdashline 0 & \blacksquare & \dots & 0 & 0  \\\hdashline \vdots & \vdots &\hspace{1mm} \ddots \  \ & 0 & 0 \\\hdashline 0 & 0 & 0 & \blacksquare & 0 \\\hdashline 0 & 0 & 0 & 0 & \begin{matrix}\vspace{-3mm}\\\hspace{1mm}\Bigblacksquare \end{matrix} \end{array} \end{bmatrix},
\end{align}
and as a result we have block-diagonalized $H$, where we have $t+1$ blocks on the diagonal. Notice that the first $t$ blocks on the diagonal are the matrices $\mathbf O^i$ for $i = i_1, \dots, i_t$ that are known to each of the agents at the leaves. Furthermore, the information needed to form $Q_1$ is distributedly known by the agents at the leaves, since we can write $Q_1$ as
\begin{multline}
Q_1 = \begin{bmatrix} \begin{array}{c:c:c:c:c:c}  I & 0 & \dots & 0 & 0 & 0 \\ \hdashline 0 & I & \dots & 0 & 0 & 0\\ \hdashline \vdots & \vdots & \ddots & \vdots & \vdots & \vdots\\ \hdashline 0 & 0 & \dots & I & 0 & 0 \\ \hdashline  -\begin{bmatrix} (\bar R_{yu}^{i_1})^T & (\bar A_{yu}^{i_1})^T \end{bmatrix}^{^{^{}} } (\mathbf O^{i_1})^{-1} & 0 & \dots & 0 & I & 0 \\ \hdashline 0 & 0 & 0 & 0 & 0 & I \end{array} \end{bmatrix} \times\\ \begin{bmatrix} \begin{array}{c:c:c:c:c:c}  I & 0 & \dots & 0 & 0 & 0 \\ \hdashline 0 & I & \dots & 0 & 0 & 0\\ \hdashline \vdots & \vdots & \ddots & \vdots & \vdots & \vdots\\ \hdashline 0 & 0 & \dots & I & 0 & 0 \\ \hdashline  0 & - \begin{bmatrix} (\bar R_{yu}^{i_2})^T & (\bar A_{yu}^{i_2})^T  \end{bmatrix}^{^{^{}} } (\mathbf O^{i_2})^{-1} & \dots & 0 & I & 0 \\ \hdashline 0 & 0 & 0 & 0 & 0 & I \end{array} \end{bmatrix} \times \\ \dots \times  \begin{bmatrix} \begin{array}{c:c:c:c:c:c}  I & 0 & \dots & 0 & 0 & 0 \\ \hdashline 0 & I & \dots & 0 & 0 & 0\\ \hdashline \vdots & \vdots & \ddots & \vdots & \vdots & \vdots\\ \hdashline 0 & 0 & \dots & I & 0 & 0 \\ \hdashline  0 & 0 & \dots & -\begin{bmatrix} (\bar R_{yu}^{i_t})^T & (\bar A_{yu}^{i_t})^T  \end{bmatrix}^{^{^{}} } (\mathbf O^{i_t})^{-1} & I & 0 \\ \hdashline 0 & 0 & 0 & 0 & 0 & I \end{array} \end{bmatrix}.
\end{multline}
This then means that not only it is possible to block-triangulate $H$ in the first stage of the upward pass as in \eqref{eq:factorKKTDiag}, but also the information that is needed to do so is distributed among the involved agents and is based on their local information. It is possible to continue this procedure by block-triangulating the last diagonal block in right hand side of \eqref{eq:factorKKTDiag} as below
\small
\begin{multline}
Q_2P_2Q_1P_1HP_1^TQ_1^TP_2^T Q_2^T =\begin{bmatrix} \begin{array}{c:c}  I \ \ & 0 \\ \hdashline 0 & \begin{matrix} \vspace{-3mm}\\\bar P_2 \end{matrix}  \end{array} \end{bmatrix} \begin{bmatrix} \begin{array}{c:c}  I  \ \ & 0 \\ \hdashline 0 & \begin{matrix} \vspace{-3mm}\\\bar Q_2 \end{matrix}  \end{array} \end{bmatrix} \begin{bmatrix} \begin{array}{c:c}  \begin{matrix} \blacksquare & & & \\ & \blacksquare & & \\ & & \hspace{1mm} \ddots \  \ & \\ & & & \blacksquare \end{matrix} & 0 \\ \hdashline 0 & \begin{matrix}\vspace{-3mm}\\\hspace{1mm}\Bigblacksquare \end{matrix}  \end{array} \end{bmatrix}\times \\\begin{bmatrix} \begin{array}{c:c}  I \ \ & 0 \\ \hdashline 0 & \begin{matrix} \vspace{-3mm}\\\bar P_2 \end{matrix}  \end{array} \end{bmatrix}^T \begin{bmatrix} \begin{array}{c:c}  I  \ \ & 0 \\ \hdashline 0 & \begin{matrix} \vspace{-3mm}\\\bar Q_2 \end{matrix}  \end{array} \end{bmatrix}^T =   \begin{bmatrix} \begin{array}{c:c}  \begin{matrix} \blacksquare & & & \\ & \blacksquare & & \\ & & \hspace{1mm} \ddots \  \ & \\ & & & \blacksquare \end{matrix} & 0 \\ \hdashline 0 & \begin{matrix}\begin{array}{c:c:c:c:c} \blacksquare & 0 & \dots & 0 & 0 \\\hdashline 0 & \blacksquare & \dots & 0 & 0  \\\hdashline \vdots & \vdots &\hspace{1mm} \ddots \  \ & 0 & 0 \\\hdashline 0 & 0 & 0 & \blacksquare & 0 \\\hdashline 0 & 0 & 0 & 0 & \begin{matrix}\vspace{-3mm}\\\hspace{1mm}\bigblacksquare \end{matrix} \end{array} \end{matrix} \end{array} \end{bmatrix}
\end{multline}
\normalsize
where similar to the previous stage $\bar P_2$ is a permutation matrix and $\bar Q_2$ is computed using a similar approach as $Q_1$. Here the newly generated small diagonal blocks are the matrices $\mathbf O^i$ with $i$ being indices of the parents of the leaves that have received all messages from their children. This step of block-diagonalization can be accomplished after the second step of the upward pass in the clique tree. We can continue this procedure upwards through the tree until we have $q$ blocks on the diagonal at which point we have arrived at the root of the tree. So having finished the upward pass we have computed
\begin{align}
\underbrace{Q_{L+1} P_{L+1} \times \dots \times Q_2 P_2 Q_1 P_1}_{L^{-1}} H \underbrace{P_{1}^T Q_{1}^T P_{2}^T Q^T_{2} \times \dots \times P^T_{L+1} Q^T_{L+1}}_{L^{-T}} = \begin{bmatrix}\begin{array}{c:c:c:c:c} \blacksquare & 0 & \dots & 0 & 0 \\\hdashline 0 & \blacksquare & \dots & 0 & 0  \\\hdashline \vdots & \vdots &\hspace{1mm} \ddots \  \ & 0 & 0 \\\hdashline 0 & 0 & 0 & \blacksquare & 0 \\\hdashline 0 & 0 & 0 & 0 & \blacksquare \end{array} \end{bmatrix}
\end{align}
with $q$ diagonal elements that are the matrices $\mathbf O^i$ for $i \in \mathbb N_q$. Notice that this means that by the end of an upward pass through the tree we have in fact computed an indefinite block $LDL^T$ factorization of $H$ where both the computation and storage of the factors are done distributedly over the clique tree.
\begin{rem}
\emph{As was shown in this section, the message-passing scheme can be viewed as a distributed multi-frontal indefinite block $LDL^T$ factorization technique that relies on fixed pivoting. This reliance is in conformance with and dictated by the structure in the problem which can in turn make the algorithm vulnerable to numerical problems that can arise, e.g., due ill-posed subproblems. Such issues can be addressed using regularization and/or dynamic pivoting strategies. Here, however, we abstain from discussing such approaches as the use of them in a distributed setting is beyond the scope of this paper.}
\end{rem}
So far we have described a distributed algorithm for computing the primal-dual directions. In the next section we put forth a distributed framework for computing proper step sizes for updating the iterates, and we will also propose a distributed method for checking the termination condition at every iteration.

\subsection{Distributed Step Size Computation and Termination}\label{sec:Step}

In this section, we first propose a distributed scheme for computing proper step sizes that relies on the approach described in Section \ref{sec:PDIPM}. This scheme utilizes, the clique tree used for calculating the primal-dual directions, for computing the step size. Similar to the message-passing procedure discussed in the previous section, in this scheme we also start the computations from the leaves of the tree. The proposed scheme comprises of two stages. During the first stage a step size bound is computed that assures primal and dual feasibility, with respect to the inequality constraints, of the iterates, and then during the second stage a back tracking line search is conducted for computing the step size which also assures persistent decrease of primal and dual residual norms. Within the first stage, let each leaf of the tree, $i$, firstly compute its bound $\bar \alpha^{i,(l+1)}$ by performing a local line search. This means that initially every agent at the leaves computes

\begin{align*}
\alpha_{\textrm{max}}^i = \minimum \left\{ 1, \minimum_{k \in \phi_i, j \in \mathbb N_{m_k}} \left\{ -\lambda_j^{k,(l)}/\Delta \lambda_j^{k,(l+1)}  \ \big | \ \Delta \lambda_j^{k,(l+1)} < 0  \right\} \right\},
\end{align*}
and then performs a local line search based on its corresponding inequality constraints, i.e., $\bar G^k$ for $k \in \phi_i$, to compute \\

\begin{algorithmic}
  \While{$\exists \ j \ \text{and} \ k \ \colon \bar G_j^k (x_{_{J_k}}^{(l)} + \alpha^{i,(l+1)} \Delta x_{_{J_k}}^{(l+1)}) > 0 $}
    \State $\bar  \alpha^{i,(l+1)} = \beta\bar  \alpha^{i,(l+1)}$
  \EndWhile
\end{algorithmic}
with $\beta \in (0,1)$ and $\bar \alpha^{i,(l+1)}$ initialized as $0.99 \alpha^i_{\textrm{max}}$. These agents will also compute the quantities
\begin{equation}
\begin{split}
p^{i,(l)}_{\textrm{norm}} &= \| r_{\text{primal}}^{i,(l)} \|^2,\\
d^{i,(l)}_{\textrm{norm}} &= \| r_{\text{dual}}^{i,(l)} \|^2,
\end{split}
\end{equation}
with $r_{\text{primal}}^{i,(l)}$ defined as in \eqref{eq:primalRes} and
\begin{align}
r_{\text{dual}}^{i,(l)} = \sum_{k \in \phi_i}\left(\nabla \bar f_i(x_{_{J_k}}^{(l)}) + \sum_{j = 1}^{m_i} \lambda_j^{k,(l)}\nabla \bar G_j^k(x_{_{J_k}}^{(l)})\right) + (\mathbf A^i)^T v^{i,(l)}´.
\end{align}
These will be used in the second stage of the step size computation. Once all leaves have computed their corresponding $\bar \alpha^{i,(l+1)}$, $p^{i,(l)}_{\textrm{norm}}$ and $d^{i,(l)}_{\textrm{norm}}$, they send these quantities to their parents where they will also conduct a similar line search and similar computations as the ones performed in the leaves. Specifically, let agent $p$ be a parent to some leaves. Then the only differences between the computations conducted by this agent and the leaves are in that the line search above is initialized as $\minimum\left\{\minimum_{k\in \children(p)}\left\{ \bar \alpha^{k,(l+1)} \right\}, 0.99 \alpha^p_{\textrm{max}}\right\}$ and that
\begin{equation}\label{eq:cliqueRes}
\begin{split}
p^{p,(l)}_{\textrm{norm}} &= \| r_{\text{primal}}^{i,(l)} \|^2 + \sum_{k \in \children(p)} p^{k,(l)}_{\textrm{norm}}, \\
d^{p,(l)}_{\textrm{norm}} &= \| r_{\text{dual}}^{i,(l)} \|^2 + \sum_{k \in \children(p)} d^{k,(l)}_{\textrm{norm}}.
\end{split}
\end{equation}
Using this procedure each agent communicates its computed $\bar \alpha^{i,(l+1)}$, $p^{i,(l)}_{\textrm{norm}} $ and $d^{i,(l)}_{\textrm{norm}}$ upwards through the tree to the root. Once the root has received all the computed step size bounds from its children/neighbors, it can then compute its local step size bound in the same manner. However, the computed bound at the root, $\bar  \alpha^{r,(l+1)}$, would then constitute a bound on the step size for updating the iterates which ensures primal and dual feasibility for the whole problem. Furthermore the computed $d^{r,(l)}_{\textrm{norm}}$ and $d^{r,(l)}_{\textrm{norm}}$ at the root will then constitute the norm of the primal and dual residuals for the whole problem computed at the iterates at iteration $l$.  This finishes the first stage of the step size computation. The second stage, is then started by communicating this step size bound downwards through the tree until it reaches the leaves. At which point each agent at the leaves computes the quantities $p^{i,(l+1)}_{\textrm{norm}}$ and $p^{i,(l+1)}_{\textrm{norm}}$ as above with the updated local iterates using the step size $\bar  \alpha^{r,(l+1)}$. These quantities are then communicated upwards through the tree to the root where each agent having received these quantities from all its children computes its corresponding $p^{i,(l+1)}_{\textrm{norm}}$ and $p^{i,(l+1)}_{\textrm{norm}}$ as in \eqref{eq:cliqueRes} using the updated local iterates. Once the root have received all information from its children it can also compute its corresponding quantities which correspond to the primal and dual residuals for the whole problem computed at the updated iterates using the step size $\bar  \alpha^{r,(l+1)}$. Then in case
\begin{align}
p^{r,(l+1)}_{\textrm{norm}} + p^{r,(l+1)}_{\textrm{norm}} > (1 - \gamma \bar \alpha^{r,(l+1)})^2\left( p^{r,(l)}_{\textrm{norm}} + p^{r,(l)}_{\textrm{norm}} \right)
\end{align}
we set $\bar \alpha^{r,(l+1)} = \beta \bar \alpha^{r,(l+1)}$ and the same procedure is repeated. However if the condition above is not satisfied, the step size computation is completed and we can choose $\alpha^{(l+1)} =\bar \alpha^{r,(l+1)}$, which is then communicated downwards through the tree until it reaches the leaves. At this point all agents have all the necessary information to update their local iterates. Notice that since all the agents use the same step size, the updated local iterates would still be consistent with respect to one another.

Having updated the iterates, it is now time to decide on whether to terminate the primal-dual iterations. In order to make this decision distributedly, we can use a similar approach as for the step size computation. Particularly, similar to the approach above, the computations are initiated from the leaves where each leaf $i$ computes the norm of its local surrogate duality gap as
\begin{equation}
\begin{split}
\hat \eta^{i,(l+1)} &= \sum_{k \in \phi_i} -(\lambda^{k,(l+1)})^T \bar G^k(x_{_{J_k}}^{(l+1)})
\end{split}
\end{equation}
The leaves then communicate these computed quantities to their corresponding parents, which will then perform the following computations
\begin{equation}
\begin{split}
\hat \eta^{p,(l+1)} &= \sum_{k \in \phi_i} -(\lambda^{k,(l+1)})^T \bar G^k(x_{_{J_k}}^{(l+1)}) + \sum_{k \in \children(p)}\hat \eta^{k,(l+1)}
\end{split}
\end{equation}
This approach is continued upwards through the tree until we reach the root. The computed quantity by the root, i.e., $\hat \eta^{r,(l+1)}$, will then be equal to the surrogate duality gap for the whole problem. This quantity together with $p^{r,(l+1)}_{\textrm{norm}}$, $d^{r,(l+1)}_{\textrm{norm}}$, which was computed during the step size computation, are used by the agent at the root to decide whether to terminate the primal-dual iterations. In case the decision is to not to terminate the iterations, then the computed surrogate duality gap is propagated downwards through the tree until it reaches the leaves of the tree, which then enables each of the agents to compute the perturbation parameter, $t$, and form their respective subproblems for the next primal-dual iteration. However, in case the decision is to terminate, then only the decision will then be propagated downwards through the tree.

By now we have put forth a distributed primal-dual interior-point method for solving loosely coupled problems. In the next section we summarize the proposed algorithm and discuss its computational properties.

\subsection{Summary of the Algorithm and Its Computational Properties}
Let us reconsider the problem in \eqref{eq:DDEOPS}. As was mentioned before, this problem can be seen as a combination of $N$ subproblems each of which is expressed by the objective function $\bar f_i$ and equality and inequality constraints defined by $\bar A^i$, $b_i$ and $\bar G^i$, respectively. Given such a problem and its corresponding sparsity graph $G_s$, in order to set up the proposed algorithm, we first need to compute a chordal embedding for the sparsity graph. Having done so, we compute the set of cliques $\mathbf C_G = \{ C_1, C_2, \dots, C_q \}$ for this chordal embedding and a clique tree over this set of cliques. With the clique tree defined, we have the computational graph for our algorithm, and we can assign each of the subproblems to a computational agent, using the guidelines discussed in Section \ref{sec:OMP}. At this point we can perform the preprocessing procedure presented in Section \ref{sec:preprocess}, if necessary, and apply our proposed distributed algorithm as summarized below to the reformulated problem.
\begin{flushleft}
\begin{algorithmic}
\State{Given $l = 0$, $\mu>1$, $\epsilon>0$, $\epsilon_{\text{feas}}>0$, $\lambda^{(0)} > 0$, $ v^{(0)}$, $x^{(0)}$ such that $\bar G^i(x_{_{J_i}}^{(0)}) \prec 0$ for all $i = 1, \dots, N$,~$\hat \eta^{(0)} = \sum_{i=1}^N -(\lambda^{i,(0)})^T\bar G^i(x_{_{J_i}}^{(0)})$ and $t = \left(\mu \sum_{i=1}^N m_i\right) /\hat \eta^{(0)}$}
\Repeat
\For {$i = 1, \dots, q$}
\State{Given $t$, $x_{_{C_i}}^{(l)}$, $ v^{i,(l)}$ and $\lambda^{k,(l)}$ for $k \in \phi_i$, agent $i$ forms its}
\State{quadratic subproblems based on its assigned objective}
\State{functions and constraints as described in \eqref{eq:QP1}--\eqref{eq:CPSQP}.}
\EndFor
\State{Perform message-passing upwards through the clique tree}
\State{Perform a downward pass through the clique tree where each agent $i$}
\State{having received optimal solutions $\Delta x_{_{S_{i\parent(i)}}}^*$,}
\State{\hspace{4mm} computes $\Delta x_{_{C_i}}^{(l+1)}$ and $\Delta v^{i,(l+1)}$ using \eqref{eq:dualvariable};}
\State{\hspace{4mm} and then computes $\Delta \lambda^{k,(l+1)}$ for all $k \in \phi_i$ using \eqref{eq:DualDirectionIn}.}
\State{Compute a proper step size, $\alpha^{(l+1)}$, by performing}
\State{upward-downward passes through the clique tree as discussed}
\State{in Section \ref{sec:Step}.}
\For {$i = 1, \dots, q$}
\State{Agent $i$ updates,}
\State  $\quad x_{_{C_i}}^{(l+1)} =  x_{_{C_i}}^{(l)} + \alpha^{(l+1)}\Delta x_{_{C_i}}^{(l+1)}$;
\State  $\quad \lambda^{k,(l+1)} =  \lambda^{k,(l)} + \alpha^{(l+1)}\Delta \lambda^{k,(l+1)}$ for all $k \in \phi_i$;
\State  $\quad  v^{i,(l+1)} =   v^{i,(l)} + \alpha^{(l+1)}\Delta  v^{i,(l+1)}$;
\EndFor
\State{Perform upward-downward pass through the clique tree to}
\State{decide whether to terminate the algorithm and/or to update}
\State{the perturbation parameter $t = \left(\mu \sum_{i=1}^N m_i\right) /\hat \eta^{(l+1)}$.}
\State {$l = l + 1$.}
\Until{the algorithm is terminated}
\end{algorithmic}
\end{flushleft}
As can be seen from the summary of the algorithm above, at each iteration of the primal-dual method we need to perform several upward-downward passes through the clique tree, one for computing the primal variables direction, one to make decision regarding terminating the algorithm and/or for updating the perturbation parameter and several for computing a proper step size. Notice that among the required upward-downward passes, the one conducted for computing the primal and dual variables directions is by far the most computationally demanding one. This is because at every run of this upward-downward pass each agent needs to form \eqref{eq:dualvariable}, which requires inverting its corresponding $\mathbf O^i$. Since primal-dual interior point methods commonly converge to a solution within 30--50 iterations, the computational burden for each agent is dominated by at most $50$ factorizations that it has to compute within the run of the primal-dual algorithm. Also notice that the required number of upward-downward passes for computing the step size, depend on the back-tracking parameters $\alpha$ and $\beta$ and it is possible to reduce this number by tuning these parameters carefully. Furthermore, for the final iterations of the primal-dual method, also known as quadratic convergence phase, there would be no need for any back-tracking operation. Let us assume that the height of the tree is equal to $L$ and that the total number of upward-downward passes that is required to accomplish the second stage of step size computations is equal to $B$. Then assuming that the primal-dual method converges within 50 iterations, the total number of upward-downward passes would mount to $B + 3\times50$ and hence the algorithm converges after $2\times L\times(B+3\times50)$ steps of message passing. Also within the run of this distributed algorithm each agent would then need to compute a factorization of a small matrix at most $50$ times and communicate with its neighbors $2\times(B+3\times50)$ times.
\begin{rem}
\emph{As was discussed in Remark \ref{rem:infeas} the primal-dual method used in this paper is an infeasible long step primal-dual method, which requires solving \eqref{eq:PD} or \eqref{eq:PDQP} only once at each iteration. However in predictor-corrector variants of primal-dual methods, computing the search directions requires solving \eqref{eq:PD} or \eqref{eq:PDQP} twice with different $r^{(l)}$ terms. This means that distributed algorithms based on message-passing that rely on predictor-corrector primal-dual methods would need two upward-downward passes to compute the search directions. However, despite the change of $r^{(l)}$, the matrices $\mathbf O^i$ formed by each agent during the upward-pass of the message-passing remains the same for both of the mentioned upward-downward passes. Consequently, each agent by caching the factorization of $\mathbf O^i$ at each iteration of the primal-dual method can significantly reduce the computational burden of the second upward-downward pass. Notice that considering the discussion in Section \ref{sec:fact}, this approach is equivalent to the caching of the factorization of the coefficient matrix of \eqref{eq:PD}}.
\end{rem}

\begin{rem}
\emph{As can be seen from the summary of the proposed algorithm, we need to initialize the algorithm with a feasible starting point, i.e., $x^{(0)}$ such that $\bar G^i(x_{_{J_i}}^{(0)}) \prec 0$ for all $i = 1, \dots, N$ and $\lambda^{(0)} > 0$. Constructing a $\lambda^{(0)}>0$ can be done independently by each agents. However, producing a suitable $x^{(0)}$ is nontrivial. In order to generate such a starting point we suggest making use of a Phase \emph{I} method based on minimizing sum of infeasibilties, \cite{boyd:04}, which entails solving the following optimization problem
\begin{subequations}\label{eq:DDEOPSPhaseI}
\begin{align}
\minimize_{S,x} \quad & \sum_{i=1}^N \mathbf 1^T s^i\\
\subject \quad & \bar G^i(E_{J_i}x) \preceq s^i,  \quad  i = 1, \dots, N,\\
& s^i \succeq -\epsilon,  \quad  i = 1, \dots, N,
\end{align}
\end{subequations}
where $\epsilon$ is a very small positive scalar and $S = (s^1, \dots, s^N)$ with $s^i \in \mathbb R^{m_i}$. In case the optimal objective value of the problem is equal to $-\epsilon\times\left( \sum_{i=1}^N m_i \right)$, then the solution $x^*$ of the problem constitutes a proper starting point for our proposed distributed algorithm. Notice that the problem in \eqref{eq:DDEOPSPhaseI} has the same coupling structure as in \eqref{eq:DDEOPS}, and hence we can use our proposed distributed algorithm, based on the same clique tree or computational graph, for computing a feasible starting point. However, for the problem in \eqref{eq:DDEOPSPhaseI}, we can easily construct a proper starting point for the algorithm. For instance, $x^{(0)} = 0$ and $s_j^{i,(0)} = \maximum(\bar G_j^i(E_{J_i}x^{(0)}), -\epsilon)$ constitute a feasible starting point, which each agent can compute independently from others.}
\end{rem}
Next we illustrate the performance of the algorithm using a numerical experiment.
\section{Numerical Experiments}\label{sec:number}
\begin{figure}[t]
\begin{center}
\includegraphics[width=6cm]{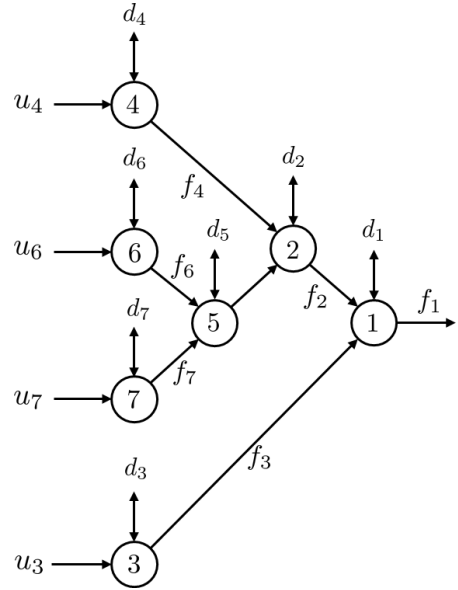}    
\caption{\small Flow problem setup \normalsize }
\label{fig:num1}
\end{center}
\end{figure}
In this section, we investigate the performance of the algorithm using an example. To this end we consider a flow problem over a tree where having received input flows from the leaves of the tree, i.e., $u_i$ for all $i \in \leaves(T)$, the collective of agents are to collaboratively provide an output flow from the root of the tree that is as close as possible to a given reference, $O_{\textrm{ref}}$. We assume that each agent $i$ in the tree produces an output flow $f_i$ that depends on the flow it receives from its children and the use of its buffer which is described using its buffer flow $d_i$, where a positive $d_i$ suggests borrowing from the buffer and a negative $d_i$ suggests directing flow into the buffer. Furthermore, there exists a cost associated with the use of the buffer and a toll for using each edge for providing flow to respective parents. The setup considered in this section is depicted in Figure \ref{fig:num1}, that is based on a tree with 7 agents. We intend to provide the requested output flow from the tree while accepting the input flow to the leaves, with minimum collective cost for the agents in the network. This problem can be formulated as

\begin{subequations}
\begin{align}
\minimize_{x} & \quad \sum_{i = 2}^q \frac{1}{2}\left( \mu_ix_i^2 + \rho_i x_{q+i}^2 \right) + \frac{1}{2}\left(\sigma\times(x_{q+1}-O_{\textrm{ref}})^2 + \mu_1 x_1^2 \right)\\
\subject & \quad \begin{rcases*} u_i + x_i = x_{q+i}\\ | x_i | \leq c_i \end{rcases*} \quad i\in \leaves(T)\\& \quad \begin{rcases*} \sum_{k \in \children(i)} x_{q+k} + x_i = x_{q+i}\\   |x_i| \leq c_i   \end{rcases*} \quad i \in \mathbb N_q \setminus \leaves(T),
\end{align}
\end{subequations}
where $x = (d_1, \dots, d_q, f_1, \dots, f_q)$ with $q = 7$, the parameters $\mu_i$, $\rho_i$ and $c_i$ denote the buffer use cost, the toll on outgoing edge and  the buffer use capacity for each agent $i$, respectively, and $\sigma$ denotes the cost incurred on the agent at the root for providing a flow that deviates from the requested output flow. Here we assume that the values of the parameters $\mu_i$, $c_i$, $\sigma$ are private information for each agent, which makes it impossible to form the centralized problem. Let us now rearrange the terms in the cost function and rewrite the problem as
\begin{subequations}\label{eq:numericalExample}
\begin{align}
\minimize_{x} & \quad \sum_{i = 2}^q \frac{1}{2}\left( \mu_ix_i^2 + \frac{\rho_i}{2} x_{q+i}^2 + \sum_{k\in\children(i)}\frac{\rho_k}{2} x_{q+k}^2 \right)\notag\\& \quad  + \frac{1}{2}\left(\sigma\times(x_{q+1}-O_{\textrm{ref}})^2 + \mu_1 x_1^2 + \sum_{k\in\children(1)}\frac{\rho_k}{2} x_{q+k}^2 \right)\\
\subject & \quad \begin{rcases*} u_i + x_i = x_{q+i}\\ | x_i | \leq c_i \\ x_{q+i} \geq 0\end{rcases*} \quad i\in \leaves(T)\\& \quad \begin{rcases*} \sum_{k \in \children(i)} x_{q+k} + x_i = x_{q+i}\\   |x_i| \leq c_i \\ x_{q+i} \geq 0  \end{rcases*} \quad i \in \mathbb N_q \setminus \leaves(T).
\end{align}
\end{subequations}
This problem can now be seen as a combination of $q = 7$ coupled subproblems where each of the subproblems is defined by each of the $q$ terms in the cost function and each of the $q$ constraint sets. The clique tree for the sparsity graph of this problem is illustrated in Figure \ref{fig:num2} and has the same structure as the flow network. As a result, this problem can be solved distributedly using the proposed message-passing algorithm while respecting the privacy of all agents.
\begin{figure}[t]
\begin{center}
\includegraphics[width=8cm]{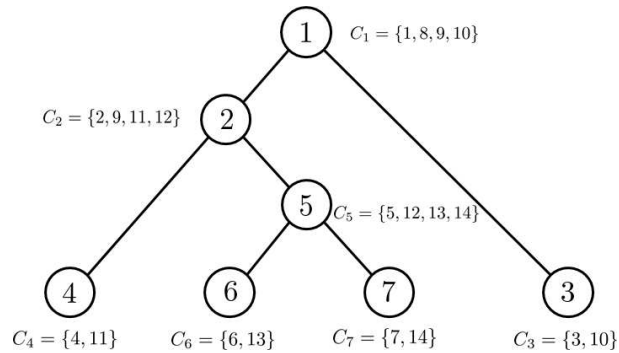}    
\caption{\small The corresponding clique tree for the sparsity graph of the flow problem \normalsize }
\label{fig:num2}
\end{center}
\end{figure}
We have solved 50 instances of the problem in \eqref{eq:numericalExample} where the parameters are chosen randomly with uniform distribution such that $u_i \in (0, 20)$, $\mu_i \in (0, 10)$, $\rho_i \in (0, 5)$, $c_i \in (0, 15)$, $O_{\textrm{ref}} \in (0, 20)$ and $\sigma \in (0, 50)$. The parameters describing the stopping criteria for all instances are chosen to be the same and are given as $\epsilon_{\text{feas}} = 10^{-8}$ and $\epsilon = 10^{-10}$, and for all cases the initial iterates are chosen to be $\lambda^{(0)} = v^{(0)} = \mathbf 1$ and $x^{(0)} = (c_1/2, \dots, c_q/2, 1, \dots, 1)$. Also the parameters used for computing the step sizes are chosen to be $\alpha = 0.05$ and $\beta = 0.5$. In the worst case the primal-dual algorithm converged after 14 iterations. The convergence behavior of the algorithm for this instance of the problem is studied by monitoring the primal and dual residuals, the surrogate duality gap and the distance to the optimal solution, as depicted in Figure \ref{fig:conv}. As expected the behavior resembles that of a primal-dual method. The optimal solution $x^*$ used for generating Figure \ref{fig:conv}-c is computed using YALMIP toolbox, \cite{lof:04}.
\begin{figure}[t]
\begin{center}
\includegraphics[width=13cm]{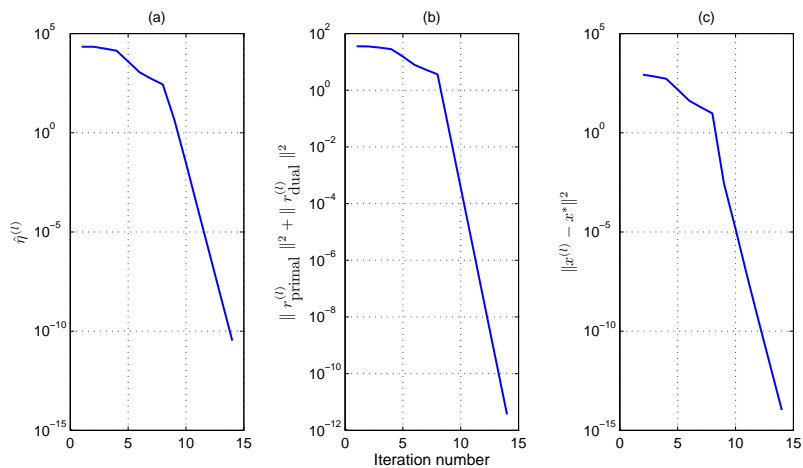}    
\caption{\small The corresponding clique tree for the sparsity graph of the flow problem \normalsize }
\label{fig:conv}
\end{center}
\end{figure}
Also the worst case total number of backtracking steps for computing step sizes was equal to 7, which was also obtained for this instance of the problem. So in total we required $2\times3\times(7 + 3\times14) = 294$ steps of message-passing to converge to the optimal solution out of which only 42 steps required agents to compute a factorization and the rest were computationally trivial. Notice that during the run of this distributed algorithm, each agent needed to compute a factorization of a small matrix only 14 times and required to communicate with its neighbors $98$ times.

We also tested the performance of the algorithm using a larger flow problem. The tree used for describing this problem was of height $L = 14$ and was generated such that all agents, except the ones at the leaves, would have two children. A tree generated in this manner then comprises of $2^{14+1} - 1 = 32767$ nodes and the problem defined on this tree has 65534 variables. The parameters that were used for defining the problem and that were used in the algorithm were chosen in the same manner as above. For this problem the primal-dual algorithm converged after 27 iterations and required a total of 21 backtracking steps. The distributed algorithm hence converged after $2856$ steps during which each agent required computing $27$ factorizations and needed to communicate with its neighbors $204$ times.
\section{Conclusion}\label{sec:conclude}
In this paper we proposed a distributed optimization algorithm based on a primal-dual interior-point method. This algorithm can be used for solving loosely coupled problems, as defined in Section \ref{sec:loose}. Our proposed algorithm relies solely on second order optimization methods, and hence enjoys superior convergence properties in comparison to other existing distributed algorithms for solving loosely coupled problems. Specifically, we showed that the algorithm converges to a very high accuracy solution of the problem after a finite number of steps that entirely depends on the coupling structure in the problem, particularly the length of the clique tree of its corresponding sparsity graph.

\appendices
\section{Proof of Theorem \ref{thm:thm2}}\label{app:app3}

We prove this theorem by induction. Firstly, note that for all neighboring agents $i$ and~$j$,
\begin{align}\label{eq:var}
V_{ij} \setminus S_{ij} = \left[  \bigcup_{k \in \Ne(i) \setminus \{ j \}} \left( V_{ki} \setminus S_{ik} \right) \right] \cup \left( C_i \setminus S_{ij} \right).
\end{align}
Moreoverو
\begin{align}\label{eq:var1}
 C_i \cap \left(V_{ki} \setminus S_{ik}\right)  = \emptyset \quad \forall \ k \in \Ne(i) ,
\end{align}
and
\begin{align}\label{eq:var2}
\left(V_{z_1i} \setminus S_{iz_1}\right) \cap \left(V_{z_2i} \setminus S_{iz_2}\right)  = \emptyset  \quad \forall \ z_1, z_2 \in \Ne(i),  \ z_1 \neq z_2,
\end{align}
where \eqref{eq:var2} is because the clique tree is assumed to satisfy the clique intersection property. These properties can also be verified for the clique tree in Figure \ref{fig:SC}. For instance let us consider agent 2 for which we have
\begin{equation}
\begin{split}
V_{21} \setminus S_{21} & = \{ 1, 3, 4, 6, 7, 8 \} \setminus \{ 1, 4\}\\
&=\left[  \bigcup_{k \in \Ne(2) \setminus \{ 1 \}} \left( V_{k2} \setminus S_{2k} \right) \right] \cup \left( C_2 \setminus S_{21} \right)\\
& = \left( V_{42} \setminus S_{24}\right) \cup \left( V_{52} \setminus S_{25}\right) \cup \left( C_2 \setminus S_{21} \right) \\
& = \left( \{ 3, 6, 7 \} \setminus \{ 3 \} \right) \cup \left( \{ 3, 8 \} \setminus \{ 3 \} \right) \cup \left( \{ 1, 3, 4 \} \setminus \{ 1 , 4 \} \right)\\
& =  \{ 6, 7 \} \cup \{ 8 \} \cup \{ 3 \},
\end{split}
\end{equation}
where as expected from \eqref{eq:var1} and \eqref{eq:var2}, the three sets making $V_{21} \setminus S_{21}$ are jointly disjoint.

We start the induction by first showing that \eqref{eq:thm2} holds for all the messages originating from the leaves of the tree, i.e., for all $i \in \leaves (T)$. This follows because for these nodes $W_{ij} = \{ i \}$ and hence $V_{ij} = C_i$. Now let us assume that $i$ is a node in the middle of the tree with neighbors $\Ne(i) = \{ k_1, \dots, k_m \}$, see Figure \ref{fig:CT}, and that
\begin{align}\label{eq:thm2a}
m_{k_ji} (x_{_{S_{k_ji}}}) = \minimum_{x_{_{V_{k_ji} \setminus S_{ik_j} }}} \left\{ \sum_{t \in \Phi_{k_ji}} \bar F_t(x_{_{J_t}}) \right\}  \quad \forall \ j = 1, \dots, m.
\end{align}
Then \eqref{eq:mij} can be rewritten as
\begin{multline}\label{eq:mijthm}
m_{ij}(x_{_{S_{ij}}}) =  \minimum_{x_{_{C_i \setminus S_{ij}}}} \huge \left\{  \sum_{t \in \phi_i} \bar F_t(x_{_{J_t}}) + \right.  \\ \left. \underbrace{\minimum_{x_{_{V_{k_1i} \setminus S_{ik_1} }}} \left\{ \sum_{t \in \Phi_{k_1i}}\bar F_t(x_{_{J_t}}) \right\}}_{m_{k_1i}} + \dots + \underbrace{\minimum_{x_{_{V_{k_m} \setminus S_{ik_m} }}} \left\{ \sum_{t \in \Phi_{k_mi}}\bar F_t(x_{_{J_t}}) \right\}}_{m_{k_mi}}\right\}.
\end{multline}
Note that $\Phi_{ij} \setminus \phi_i = \bigcup _{t \in \Ne(i) \setminus \{ j \} }\Phi_{ti}$ with $\Phi_{z_1i} \cap \Phi_{z_2i} = \emptyset,   \ \forall z_1, z_2 \in \Ne(i)\setminus \{ j \}, \ z_1 \neq z_2$. This is guaranteed since each component of the objective function is assigned to only one agent. Then by \eqref{eq:var1} and \eqref{eq:var2} we have
\begin{align}\label{eq:thm2b}
m_{ij}(x_{_{S_{ij}}}) =  \quad \minimum_{x_{_{C_i \setminus S_{ij}}}}\minimum_{x_{_{V_{k_1i} \setminus S_{ik_1} }}}\dots \minimum_{x_{_{V_{k_mi} \setminus S_{ik_m} }}} \left\{ \sum_{t \in \Phi_{ij}}\bar F_t(x_{_{J_t}}) \right\}.
\end{align}
Now we can merge all the minimum operators together and, by \eqref{eq:var}, rewrite \eqref{eq:thm2b} as
\begin{align}\label{eq:thm2c}
m_{ij} (x_{_{S_{ij}}}) = \minimum_{x_{_{V_{ij} \setminus S_{ij} }}} \left\{  \sum_{t \in \Phi_{ij}}\bar F_t(x_{_{J_t}}) \right\},
\end{align}
which completes the proof.
\section{Proof of Theorem \ref{thm:thm3}}\label{app:app4}

Using Theorem \ref{thm:thm2}, we can rewrite \eqref{eq:RLocalProblem} as
\begin{multline}\label{eq:mijthm3}
x^\ast_{_{C_r}} =  \argmin_{x_{_{C_r}}}  \left\{  \sum_{k \in \phi_r} \bar F_k(x_{_{J_k}})  + \underbrace{\minimum_{x_{_{V_{k_1r} \setminus S_{rk_1} }}} \left\{ \sum_{t \in \Phi_{k_1r}}\bar F_t(x_{_{J_t}}) \right\}}_{m_{k_1r}} + \dots \right. \\ \left. \hspace{28mm}+ \underbrace{\minimum_{x_{_{V_{k_rr} \setminus S_{rk_r} }}} \left\{  \sum_{t \in \Phi_{k_rr}}\bar F_t(x_{_{J_t}}) \right\}}_{m_{k_rr}}\right\}.
\end{multline}
where we have assumed that $\Ne(r) = \{ k_1, \dots, k_r \}$. Note that $\mathbb N_n \setminus C_r = \bigcup_{k \in \Ne(r)} V_{kr} \setminus S_{rk}$. Then by \eqref{eq:var2} we can push the \emph{minimum} operators together and rewrite~\eqref{eq:mijthm3} as
\begin{align}\label{eq:mijthm3a}
x^\ast_{_{C_r}} = \argmin_{x_{_{C_r}}}  \left\{  \sum_{k \in \phi_r} \bar F_k(x_{_{J_k}})  + \minimum_{x_{_{\mathbb N_n\setminus C_r }}} \left\{ \sum_{k \in \Ne(r)} \sum_{t \in \Phi_{kr}}\bar F_t(x_{_{J_t}}) \right\} \right\}.
\end{align}
Moreover, since $\mathbb N_N \setminus \phi_r = \bigcup_{k \in Ne(r)} \Phi_{kr}$ and that $\bigcup_{k \in \phi_r} J_k \subseteq C_r$, we can further simplify~\eqref{eq:mijthm3a} as
\begin{align}\label{eq:LocalProblem1a}
x^\ast_{_{C_r}} = \argmin_{x_{_{C_r}}} \left\{  \minimum_{x_{_{\mathbb N_n \setminus C_r}}} \left\{ \bar F_1(x_{_{J_1}}) + \dots, \bar F_N(x_{_{J_N}})\right\}\right\},
\end{align}
which completes the proof.

\newpage


\end{document}